\documentclass[12pt]{amsart}
\usepackage[pdftex]{hyperref}
\hypersetup{colorlinks,citecolor=blue,linktocpage,hyperindex=true,backref=true}
\usepackage{amsmath,amscd,amssymb,euscript}
\usepackage{mathtools}
\usepackage{latexsym}
\usepackage{graphicx}
\usepackage{tikz}
\usetikzlibrary{backgrounds}
\usetikzlibrary{decorations.fractals}
\usetikzlibrary{calc,intersections,through,backgrounds}
\usepackage{mathrsfs}
\usepackage{epstopdf}
\usepackage{rotating}
\usepackage{lscape}
\input xy
\xyoption{all}

\newcommand{\bC}{{\mathbb C}}
\newcommand{\bF}{{\mathbb F}}
\newcommand{\bG}{{\mathbb G}}
\newcommand{\bH}{{\mathbb H}}
\newcommand{\bK}{{\mathbb K}}

\newcommand{\bP}{{\mathbb P}}
\newcommand{\bQ}{{\mathbb Q}}

\newcommand{\bW}{{\mathbb W}}
\newcommand{\bZ}{{\mathbb Z}}
\newcommand{\cA}{{\mathcal A}}

\newcommand{\cC}{{\mathcal C}}
\newcommand{\cD}{{\mathcal D}}

\newcommand{\cH}{{\mathcal H}}

\newcommand{\cM}{{\mathcal M}}

\newcommand{\cO}{{\mathcal O}}
\newcommand{\cP}{{\mathcal P}}

\newcommand{\cR}{{\mathcal R}}
\newcommand{\cT}{{\mathcal T}}
\newcommand{\cU}{{\mathcal U}}
\newcommand{\cV}{{\mathcal V}}

\newcommand{\cY}{{\mathcal Y}}

\newcommand{\oB}{\overline{B}}

\newcommand{\ocJ}{{\overline{\mathcal J}}}
\newcommand{\omu}{{\overline{\mu}}}
\newcommand{\ocM}{{\overline{\mathcal M}}}

\newcommand{\ocR}{{\overline{\mathcal R}}}

\newcommand{\ra}{\rightarrow}
\newcommand{\lra}{\longrightarrow}

\newcommand{\inj}{\hookrightarrow}
\newcommand{\surj}{\mathrel{\mathrlap{\rightarrow}\mkern1mu\rightarrow}}

\newcommand{\G}{\Gamma}

\newcommand{\T}{\Theta}
\newcommand{\tT}{\widetilde{\Theta}}

\newcommand{\oT}{\overline{\Theta}}

\newcommand{\Coker}{\operatorname{Coker}}

\newcommand{\im}{\operatorname{im}}
\newcommand{\Img}{\operatorname{Im}}
\newcommand{\Ker}{\operatorname{Ker}}

\newcommand{\Pic}{\operatorname{Pic}}
\newcommand{\Prym}{\operatorname{Prym}}

\theoremstyle{definition}

\newtheorem{proposition}{Proposition}[section]
\newtheorem{lemma}[proposition]{Lemma}

\newtheorem{theorem}[proposition]{Theorem}

\newtheorem{corollary}[proposition]{Corollary}

\newtheorem{remark}[proposition]{Remark}

\numberwithin{equation}{section}

\begin{document}

\title[The irreducible components of the primal cohomology]{The irreducible components of the primal cohomology of the theta divisor of an abelian fivefold}

\author{Elham Izadi}

\address{Department of Mathematics, University of California San Diego, 9500 Gilman Drive \# 0112, La Jolla, CA 92093-0112, USA}

\email{eizadi@math.ucsd.edu}

\author{Jie Wang}

\address{Department of Mathematics, University of California San Diego, 9500 Gilman Drive \# 0112, La Jolla, CA 92093-0112, USA}
\email{jiewang884@gmail.com}

\thanks{The authors are indebted to the referee for many useful suggestions that have helped improve the exposition of this paper. The first author was partially supported by the National
Science Foundation. Any opinions, findings
and conclusions or recommendations expressed in this material are those
of the authors and do not necessarily reflect the views of the
NSF}


\begin{abstract}

  The primal cohomology $\bK_\bQ$ of the theta divisor $\T$ of a principally
  polarized abelian fivefold (ppav) is the direct sum of its invariant and anti-invariant parts $\bK_\bQ^{+1}$, resp. $\bK_\bQ^{-1}$ under the action of $-1$. For smooth $\T$, these have dimension $6$ and $72$ respectively. We show that $\bK_\bQ^{+1}$ consists of Hodge classes and, for a very general ppav, $\bK_\bQ^{-1}$ is a simple Hodge structure of
  level $2$. 

\end{abstract}

\maketitle

\tableofcontents

\section*{Introduction}

Let $A$ be a principally polarized abelian variety of dimension $g \geq 4$ with smooth theta divior $\T$. By the Lefschetz hyperplane theorem and Poincar\'e Duality (see, e.g., \cite{IzadiWang2015}) the cohomology of $\Theta$ is determined by that of $A$ except in the middle dimension $g-1$. The primitive cohomology of $\Theta$, in the sense of Lefschetz, is
\[
H^{g-1}_{pr}(\T, \bZ):=\Ker \left( H^{g-1}(\T,\bZ)\stackrel{\cup \theta|_\T}\lra H^{g+1}(\T,\bZ)\right).
\]
The primal cohomology of $\Theta$ is defined as (see \cite{IzadiWang2015} and \cite{IzadiTamasWang})
\[
\bK:=\Ker (j_*: H^{g-1}(\T,\bZ)\lra H^{g+1}(A,\bZ))
\] 
where $j : \Theta \ra A$ is the inclusion. This is a Hodge substructure of $H^{g-1}_{pr}(\T,\bZ)$ of rank $g! - \frac{1}{g+1}{2g \choose g}$ and level $g-3$ while the primitive cohomology $H^{g-1}_{pr}(\T,\bZ)$ has full level $g-1$.

The primal cohomology is therefore a good test case for the general Hodge conjecture. The general Hodge conjecture predicts that $\bK_{\bQ} := \bK \otimes \bQ$ is contained in the image, via Gysin push-forward, of the cohomology of a smooth (possibly reducible) variety of pure dimension $g-3$ (see \cite{IzadiWang2015}). This conjecture was proved in \cite{IzadiVanStraten} and \cite{IzadiTamasWang} in the cases $g=4$ and $g=5$. When $g=4$, it also follows from the proof of the Hodge conjecture in \cite{IzadiVanStraten} that for $(A,\Theta)$ generic, $\bK$ is a simple Hodge structure (isogenous to the third cohomology of a smooth cubic threefold). In this case the primal cohomology is fixed under the action of $-1$.

In the case $g=5$, the action of $-1$ splits $\bK_\bQ$ into the direct sum of its invariant piece $\bK_\bQ^{+1}$ and its anti-invariant piece $\bK_\bQ^{-1}$. We show that these have respective dimensions $6$ and $72$. Furthermore, we show that $\bK^{3,1}$ is contained in $\bK_\bC^{-1}$ which shows that $\bK_\bQ^{+1}$ consists of Hodge classes. It follows from the main result of \cite{IzadiTamasWang} and the Lefschetz $(1,1)$ theorem that the latter are algebraic. In the upcoming paper \cite{ConderDeweyIzadi}, we will concretely describe natural surfaces in $\T$ whose classes generate $\bK^{+1}$ and describe the structure of the lattice $\bK^{+1}$ with the bilinear form induced by the cup product on $H^4 (\T, \bZ)$.

Our main result here is

\begin{theorem}\label{thmmain}
For a very general ppav $A$ of dimension $5$ with smooth theta divisor $\T$, the anti-invariant primal cohomology $\bK^{-1}$ of $\T$ is a simple Hodge structure of level $2$. 
\end{theorem}

Here of course, by simple we mean that the only Hodge substructures of $\bK^{-1}$ are $\{0\}$ and those of finite index. The above theorem is equivalent to the fact that the only Hodge substructures of $\bK^{-1}_{\bQ}$ are $\{0\}$ and itself. It is this latter statement that we will prove.

The above theorem simplifies the proof of the Hodge conjecture in \cite{IzadiTamasWang}: it is no longer necessary to show that the image of the Abel-Jacobi map in \cite{IzadiTamasWang} contains all of $\bK_{\bQ}$, only that it contains $\bK^+_\bQ$ and intersects $\bK^-_{\bQ}$ non-trivially.

If $A$ is replaced by a projective space and $\Theta$ by a smooth hypersurface, then the primitive and the primal cohomology coincide. The primitive cohomology of a general hypersurface is simple (see, e.g., \cite[7.3]{Lamotke-Lefschetz-1981}).

Our strategy, explained below, for proving Theorem \ref{thmmain} is to use the Mori-Mukai proof \cite{MoriMukai83} of the unirationality of $\cA_5$.

Let $T$ be an Enriques surface and
\[
f: S \lra T
\]
the K3 \'etale double cover corresponding to the canonical class (which is $2$-torsion) $K_T\in \Pic(T)$.
Let $H$ be a very ample line bundle on $T$ with $H^2=10$. A general element in the linear system $|H|\cong\bP^5$ is a smooth curve of genus $6$ and such smooth curves are parametrized by the Zariski open subset $|H|\setminus \cD$, where $\cD$ is the dual variety of the embedding of $T$ in $|H|^*$. For each element $u\in |H|\setminus \cD$, we obtain a nontrivial \'etale double cover $D_u:=f^{-1}(C_u)\ra C_u$. Associating to such a cover its Prym variety $P(D_u,C_u)$ defines a morphism from $|H|\setminus \cD$ to $\cA_5$:
\[
\xymatrix{|H|\setminus \cD\ar[rr]^{\mu_H}\ar[rd]_-{\cP_H}&&\cR_6\ar[ld]^-{\mathcal{P}}\\
&\cA_5&}.
\]
The linear systems $|H|$ form a projective bundle $\cH$ over the moduli space of Enriques surfaces and the maps $\mu_H$ are restrictions of a rational map
\begin{equation}\label{eqmu}
\mu : \cH \lra \cR_6.
\end{equation}
Mori and Mukai \cite{MoriMukai83} showed that $\mu$ is dominant.

The ppav $(A,\T)$ with singular theta divisor form the Andreotti-Mayer divisor $N_0$ in $\cA_5$ (\cite{Beauville1977-1}). The divisor $N_0$ has two irreducible components $\theta_{null}$ and $N_0'$ (\cite{Debarre921},\cite{MumfordKodAg-1983})) (as divisors, $N_0=\theta_{null}+2N_0'$). The theta divisor of a general point  $(A,\T)\in\theta_{null}$  has a unique node at a two-torsion point while the theta divisor of a general point in $N'_0$ has two distinct nodes $x$ and $-x$. 

The primal cohomologies of the theta divisors form a variation of (polarized) Hodge structures over $\cU:=|H|\setminus (\cD\cup \overline{\cP_H^{-1}(N_0)})$. Inspired by \cite[7.3]{Lamotke-Lefschetz-1981}, we prove Theorem \ref{thmmain} via a detailed study of the monodromy representation
\[
\rho:\pi_1(\cU)\lra Aut(\bK^{-1}_\bQ,\langle,\rangle)
\]
where $\langle, \rangle$ is the natural polarization on $\bK^{-1}_{\bQ}$ induced by the intersection pairing on $H^4 (\Theta, \bQ)$.

Our argument has two mainly independent parts.

In the first part, starting from a Lefschetz pencil $l\cong\bP^1\subset |H|$ of curves of degree $10$ in $T$, we construct a family of Prym varieties with their theta divisors (Section \ref{secpencil})
\[
\xymatrix{\Theta\ar[r]\ar[rd]&A\ar[d]\\
&l.}
\]
We show that the double covers of the singular curves of the pencil are irreducible (Proposition \ref{propirreducible}) and use this to compactify the semi-abelian Prym varieties of these double covers and their theta divisors (Paragraph \ref{subseccompPrym}). We begin Section \ref{secsmooth} by computing the number (counted with multiplicities) of smooth curves in our pencil whose Prym varieties have singular theta divisors. These singular theta divisors have at worst two ordinary double points. In Lemma \ref{lemmult1} we show that all these fibers occur with multiplicity $1$ in the family. We then combine this, in Proposition \ref{proptotsmooth}, with a calculation at the level of tangent spaces to the singular Pryms of our family, to show that the total spaces $A$ and $\Theta$ are smooth.

The second part of our argument consists of the study of the local monodromy $N_i$ on
\[
\bK_t^{-1} := \bK^{-1}_{\bQ,t} \subset H^{g-1} (\Theta_t, \bQ)
\]
for $t$ near a point $p_i\in l$ where the theta divisor $\Theta_{p_i}$ is singular. We first show in Section \ref{secN0monodromy} that, when $p_i\in N_0$, the monodromy around $p_i$ is an involution. To study the monodromy when the Prym variety at $p_i$ degenerates, we use the Clemens-Schmid sequence. In Section \ref{secClemensSchmid} we recall the basic facts that we need about the Clemens-Schmid sequence. Next, in Section \ref{secboundary} we compute the graded pieces of the limit mixed Hodge structures on $H^4 (A_t), H^4 (\T_t)$ and $\bK_t$. In particular, we show that the limit mixed Hodge structure $\bK_{\lim}^{i,-1}$ on $\bK_t^{-1}$ contains a copy $\bH_i = N_i (\bK_t)$ of the primal cohomology of the theta divisor of a general abelian fourfold which is simple by \cite{IzadiVanStraten}. In Section \ref{sec-Id} we compute the action of $-Id$ on the primal cohomology and show that $\bK_\bQ^{+1}$ is contained in the kernel of $N_i$. We also show that, for all $g$, the Hodge structure $\bK^{(-1)^{g-1}}$ has level $g-5$ (Corollary \ref{corg-5}). In Section \ref{secglobalmono} we prove our main theorem. We show how it essentially follows from the global invariant cycles theorem that $\cap_{i=1}^{42} \Ker N_i = \{0\}$ (Lemma \ref{lemker}) and deduce from this that $\bK_t^{-1} = \sum_{i=1}^{42} \bH_i$ (Corollary \ref{corprim}). Now, given a nontrivial Hodge substructure $\bF_t$ of $\bK_t^{-1}$, our argument can be summarized as follows
\[
0 \neq x\in \bF_t \Longrightarrow \exists \; i \; \colon 0 \neq N_i (x) \in N_i (\bK_t^{-1}) \cap \bF_t = \bH_i \cap \bF_t \Longrightarrow \bH_i \subset \bF_t.
\]
We show that $\bF_t$ is globally invariant under the action of the monodromy group (end of Section \ref{secglobalmono}) which permutes the $\bH_i$ transitively (Lemma \ref{lemconjugate}), therefore all the $\bH_i$ and hence $\bK_t^{-1}$ are contained in $\bF_t$.

\section{Prym varieties associated to a Lefschetz pencil }\label{secpencil}

\subsection{A pencil of double covers}\label{subsecpencil}


We denote by
\[
\tau:S\lra S
\]
the fixed point free covering involution such that $S/\tau\cong T$. By \cite[Prop. 2.3] {Namikawa-Enriques-1985} the invariant subspace of the involution $\tau^*$ acting on the N\'eron Severi group $NS(S)$ is equal to $f^*(NS(T))$. Since the pullback
\[
f^*: NS(T)\lra NS(S)
\]
is injective, we deduce that $f^*(NS(T))$ is a rank $10$ primitive sublattice in $NS(S)$. It follows that the Picard number of $S$ is greater than or equal to $10$. By \cite[Prop. 5.6] {Namikawa-Enriques-1985}, when $T$ is general in moduli,
\begin{eqnarray}\label{general}NS(S)=f^*NS(T).
\end{eqnarray}
{\bf Hypothesis:} Throughout this paper, we will assume $T$ satisfies (\ref{general}).

Suppose $l\cong\bP^1\subset |H|$ is a Lefschetz pencil, i.e., it is transverse to the dual variety $\cD$. Then the singular curves of the pencil consist of finitely many irreducible nodal curves. Denote by $\widetilde{T}:=Bl_{10}T$ (resp. $\widetilde{S}:=Bl_{20}S$) the blow-up of $T$ (resp. $S$) along the base locus of $l$ (resp. $f^*l$). We obtain a family of \'etale double covers parametrized by $l$:
\[
\xymatrix{\widetilde{S}\ar[rr]^{\widetilde{f}}\ar[rd]_{\pi'}&&
\widetilde{T}\ar[ld]^-{\pi}\\
&l.&}
\]
\begin{proposition}\label{propsingular}
There are $42$ singular fibers in the family $\widetilde{T}\stackrel\pi\lra l$.
\end{proposition}
\begin{proof}
We use the formula
\[
\chi_{top}(\widetilde{T})=\chi_{top}(T)+10=\chi_{top}(\bP^1)\chi_{top}(C)+N,
\]
where $C$ is a smooth fiber in the pencil and $N$ is the number of singular fibers. Since $\chi_{top} (T) = 12$ and $\chi_{top} (C) = -10$, we obtain $N=42$.
\end{proof}

Denote by $C_t$ the fiber over $t\in l$ of $\pi$ and $D_{t}$ the corresponding \'etale double cover in $\widetilde{S}$ and $\{s_i\in l:i=1,...,42\}$ the $42$ points where $\pi$ is singular.

\begin{proposition}\label{propirreducible}
For any $t\in l$, the \'etale double cover $D_t$ of $C_t$ is an irreducible curve. 
\end{proposition}
\begin{proof} Suppose $D_t$ is reducible for some $t$.  If $C_t$ is smooth, $D_t$ must be the trivial cover. If $C_t$ has one node, $D_t$ is either the trivial cover or the Wirtinger cover. 
 In either case, the involution $\iota$ permutes the two components $D_t^1$ and $D_t^2$ of $D_t$. By (\ref{general}), the class of $D_t^i$ in $NS(S)$ is $\iota$ invariant, thus $D_t^1$ and $D_t^2$ have the same class in $NS(S)$ and $H = 2 D_t^1$.  However, since $H^2=10$, the class of $H$ in $NS(T)$ is not $2$-divisible, a contradiction.
\end{proof}

\begin{corollary}\label{corcover}
For a singular fiber $C_{s_i}=C_{pq}:=\frac{C}{\{p\sim q\}}$ in the pencil $l$, the \'etale double cover $D_{s_i}:=D_{pq}$ is obtained by glueing $p_i$ with $q_i$ for $i=1,2$ on a nontrivial \'etale double cover $D$ of $C$, where $p_i,q_i\in D$ are the inverse images of $p,q\in C$ respectively.
\end{corollary}
\begin{proof}The \'etale double cover $D_{pq}$ of $C_{pq}$ is determined by a $2$-torsion point in $\Pic^0(C_{pq})$. The statement follows immediately from the irreducibility of $D_{s_i}$ and the exact sequence
\[
\xymatrix{1\ar[r]&\bZ_2\ar[r]&\Pic^0(C_{pq})_2\ar[r]^-{\nu^*}&\Pic^0(C)_2\ar[r]&0},\]
where $\nu:C\ra C_{pq}$ is the normalization map and the kernel of $\nu^*$ is generated by the point of order $2$ corresponding to the Wirtinger cover.
\end{proof}

\subsection{The compactified Prym variety}\label{subseccompPrym}
We describe the compactified Prym variety for the cover $D_{pq}\ra C_{pq}$ as in Corollary \ref{corcover}. The semiabelian part $G_{pq}$ of the Prym variety is the identity component $\Ker^0(Nm_{pq})$ of $\Ker(Nm_{pq})\subset \Pic^0(D_{pq})$ in the following commutative diagram with exact rows and columns
\[
\xymatrix{& 0\ar[d] & 0\ar[d] & 0\ar[d] \\
1\ar[r]&\bC^*\ar[r]\ar[d]&\Ker(Nm_{pq})\ar[r]\ar[d]&\Ker(Nm)\ar[r]\ar[d]&0\\
1\ar[r]&(\bC^*)^2\ar[d]\ar[r]&\Pic^0(D_{pq})\ar[d]^-{Nm_{pq}}\ar[r]&\Pic^0(D)\ar[r]\ar[d]^-{Nm}&0\\
1\ar[r]&\bC^*\ar[r]\ar[d]&\Pic^0(C_{pq})\ar[r]\ar[d]&\Pic^0(C)\ar[r]\ar[d]&0\\
 & 0 & 0 & 0.}
\]
It follows immediately that the group scheme $G_{pq}$ is a $\bC^*$-extension of the Prym variety $(B,\Xi):=\Prym(D,C)$:
\[
\xymatrix{1\ar[r]&\bC^*\ar[r]&G_{pq}\ar[r]&B\ar[r]&0}.
\]
Let $p:P^\nu\ra B$ be the unique $\bP^1$ bundle containing $G_{pq}$ and write $P^\nu\setminus G_{pq}=B_0\amalg B_\infty$, where $B_0$ and $B_\infty$ are the zero and infinity sections of $P^\nu$.

The compactified `rank one degeneration' $P$ is constructed as follows (c.f. \cite[\S 1]{MumfordKodAg-1983}). 

\begin{enumerate}
\item On $P^\nu$, we have the linear equivalence $B_0-B_\infty\backsim_{lin}p^{-1}(\Xi-\Xi_b)$ for a unique $b\in B$. Thus
\[
B_0+p^{-1}\Xi_b\backsim_{lin}B_\infty+p^{-1}\Xi.
\]

\item Let $L^\nu:=\cO_{P^\nu}(B_0+p^{-1}\Xi_b)$. Then $L^\nu|_{B_0}\cong\cO_B(\Xi)$ and $L^\nu|_{B_\infty}\cong \cO_B(\Xi_b)$. Via the Leray spectral sequence for $p$, we see that $h^0(P^\nu,L^\nu)=2$ and $B_0+p^{-1}\Xi_b$, $B_\infty+p^{-1}\Xi$ span $|L^\nu|$.

\item The compactified Prym variety ${P}$ is constructed from $P^\nu$ by identifying the zero section $B_0 \stackrel{p}{\cong} B$ with the infinity section $B_\infty \stackrel{p}{\cong} B$ via translation by $b\in B$. We also denote by $\nu : P^{\nu} \ra P$ the normalization morphism.

\item The line bundle $L^\nu$ descends to a line bundle $L$ on $P$, i.e., $\nu^* L \cong L^{\nu}$. The linear system $|L|$ is a point.

\item The theta divisor $\Upsilon\subset P$ is the unique divisor in $|L|$.

\end{enumerate}

\begin{remark}\label{rktrivial}The $\bP^1$ bundle $P^\nu\ra B$ contains an open subset $P^\nu\setminus B_\infty$ (resp. $P^\nu\setminus B_0$), which is isomorphic to the total space of $N_{B_0|P^\nu}\cong\cO_{B_0}(B_0)\cong\cO_B(\Xi-\Xi_b)$ (resp. $\cO_B(\Xi_b-\Xi)$). We conclude that $P^\nu\cong \bP_B(\cO_B(\Xi-\Xi_b)\oplus\cO_B(\Xi_b-\Xi))$. In particular $G_{pq}\ra B$ and $P^\nu\ra B$ are {\bf topologically trivial} $\bC^*$ and $\bP^1$ bundles, respectively.
\end{remark}

\begin{proposition}\label{propdegtheta}For a general rank one degeneration, the normalization $\Upsilon^\nu$ of the theta divisor is isomorphic to $Bl_{\Xi\cap\Xi_b}B\subset P^\nu$, the theta divisor $\Upsilon\subset P$ is obtained from $\Upsilon^\nu$ by identifying the proper transforms of $\Xi$ and $\Xi_b$.
\end{proposition}

\begin{proof} Let $\sigma_0,\ \sigma_\infty$ be elements of $H^0(P^\nu,L^\nu)$, such that $div(\sigma_0)= B_0+p^{-1}\Xi_b$ and $div(\sigma_\infty)= B_\infty+p^{-1}\Xi$.  After rescaling, we may assume, under the natural identification $B_0 \stackrel{p}{\cong} B \stackrel{p}{\cong} B_\infty$, that $\sigma_0|_{B_\infty}$ and $\sigma_\infty|_{B_0}$ differ by translation by $b$. Then $\sigma_0+\sigma_\infty$ descends to a section of $L$. Since $(\sigma_0+\sigma_\infty)|_{B_0}$ vanishes precisely on $\Xi$ and $(\sigma_0+\sigma_\infty)|_{B_\infty}$ vanishes precisely on $\Xi_b$, we conclude that for $u\in B\setminus(\Xi\cap\Xi_b)$, $0\ne(\sigma_0+\sigma_\infty)|_{p^{-1}(u)}\in H^0(\cO_{\bP^1}(1))$.  Thus $\Upsilon^\nu:=div(\sigma_0+\sigma_\infty)$ maps one-to-one to $B$ away from $\Xi\cap\Xi_b$. On the other hand, the base locus of the pencil $|L^\nu|$ is clearly $p^{-1}(\Xi\cap\Xi_b)$. Thus $\Upsilon^\nu=m[Bl_{\Xi\cap\Xi_b}B]$, for some integer $m$, as divisors in $P^\nu$. Since $(\sigma_0+\sigma_\infty)|_{B_0}$ is reduced, $m=1$.
\end{proof}

To summarize, we have the family of (compactified) Prym varieties and theta divisors
\[
\T\lra A\lra l.
\]

\section{The smoothness of the total spaces $A$ and $\Theta$}\label{secsmooth}

Our main goal in this section is to prove that the total spaces $A$ and $\Theta$ of the family of compactified Prym varieties constructed in the previous section are smooth. We begin with some numerical computations, determining the numbers of singular theta divisors in the family.

The family of compactified Prym varieties defines a morphism $\rho : l\ra \tilde{\cA}_5$ where  $\tilde{\cA}_5$ is the partial compactification of $\cA_5$ parametrizing ppav $(A,\T)$ of dimension $5$ and their rank $1$ degenerations. This space is a quasi-projective variety and is essentially the blow-up of the open set $\cA_5\amalg\cA_4$ in the Satake-Baily-Borel compactification $\cA_5^*$ along its boundary $\cA_4$ (\cite{Igusa-DesingSieg-1967}). The coarse moduli space of $\tilde\cA_5$ is the union of $\cA_5$ and a divisor $\Delta$ parametrizing rank $1$ degenerations. Mumford \cite{MumfordKodAg-1983} computed the class of the closure of $\theta_{null}$ and $N_0'$ in $\tilde{A}_5$ to be

\begin{eqnarray}\label{eqnANdivisor1}
[\theta_{null}]=264\lambda-32\delta,
\end{eqnarray}
\begin{eqnarray}\label{eqnANdivisor2}
[N'_0]=108\lambda-14\delta,
\end{eqnarray}
\begin{eqnarray}\label{eqnANdivisor3}
[ N_0]=[\theta_{null}]+2[N_0']=480\lambda-60\delta,
\end{eqnarray}
where $\lambda$ is the first Chern class of the Hodge bundle $\Lambda$ and $\delta$ is the class of $\Delta$.
\begin{lemma}\label{lemlambda}
The degree of $\rho^* \lambda$ is $6$.
\end{lemma}
\begin{proof}The pull-back of the Hodge bundle $\Lambda$ to $l$ fits in the exact sequence
\[
0\lra\pi_*\omega_{\widetilde{T}/l}\lra\pi'_*\omega_{\widetilde{S}/l}\lra \rho^* \Lambda \lra0,
\] 
where $\omega_{\widetilde{T}/l}$ and $\omega_{\widetilde{S}/l}$ are the relative dualizing sheaves.
Thus $c_1(\rho^*\Lambda)=c_1(\pi'_*\omega_{\widetilde{S}/l})-c_1(\pi_*\omega_{\widetilde{T}/l})$.
 We directly compute that the relative dualizing sheaf $\omega_{\widetilde{T}/l}=K_{\widetilde{T}}\otimes \pi^*K_l^{-1}$ has self intersection number
 $(\omega_{\widetilde{T}/l})^2=30$. Since, by Proposition \ref{propsingular}, the number of singular curves in the pencil is $42$, applying Mumford's relation \cite[Chapter 13.7]{ACG} on $\overline{M}_6$, we see that $c_1(\pi_*\omega_{\widetilde{T}/l})=\frac{30+42}{12}=6$. Similarly, we compute $c_1(\pi'_*\omega_{\widetilde{S}/l})=12$ and therefore $c_1 (\rho^*\Lambda)=6$.
\end{proof}

\begin{corollary}\label{corN0}
In the pencil $l$, counting with multiplicities, there are $240$ fibers with theta divisor singular at a unique two-torsion point and $60$ fibers with theta divisor singular at two points.
\end{corollary}
\begin{proof} Since $l \cdot \delta = 42$ by Proposition \ref{propsingular} and $l \cdot \lambda = 6$ by Lemma \ref{lemlambda}, we directly compute $l\cdot[\theta_{null}]=l\cdot(264\lambda-32\delta)=240$ and $l\cdot[N_0']=l\cdot(108\lambda-14\delta)=60$.
\end{proof}

To prove the smoothness of the total spaces $A$ and $\T$, we first need two lemmas.\\
Let $\theta_{null}'$ be the divisor of $\cR_6$ parametrizing double covers of curves of genus $6$ with a vanishing theta-characeristic.

\begin{lemma}\label{lemgenthetanull}
The image of $\mu : \cH \ra \cR_6$ (see \eqref{eqmu}) contains a non-empty open subset of $\theta_{null}'$. 
\end{lemma}

\begin{proof}
The inverse image $\mu^{-1}\theta_{null}'$ is a divisor in $\cH$ which is non-empty by Corollary \ref{corN0}. Hence, if the image of $\mu$ does not contain a non-empty open subset of $\theta_{null}'$, the image $\mu(\mu^{-1}\theta_{null}')$ is contained in a divisor in $\theta_{null}'$ and $\mu$ is not quasi-finite anywhere on $\mu^{-1}\theta_{null}'$. We now show that there is a locus in $\mu^{-1}\theta_{null}'$ where $\mu$ is quasi-finite, thus obtaining a contradiction.

Let $\ocR_6$ be the Deligne-Mumford compactification of $\cR_6$, as constructed in \cite[p. 179]{Beauville1977-1} (where it is denoted $\ocJ$). By construction, the natural projection $\cR_6 \ra \cM_6$ extends to a proper morphism $q : \ocR_6 \ra \ocM_6$. The extended rational map $\omu : \cH \ra \cR_6 \inj \ocR_6$ is well-defined on all nodal curves in $\cH$. It follows from \cite[Proposition 3.2, p. 349]{MoriMukai83} that, for any Enriques $T$ and any stable curve $C= E \cup F\in |H|$ where $E$ is an elliptic curve and $F$ is a smooth curve of genus $3$ meeting $E$ at $3$ distinct points, the map $\omu$ is quasi-finite at $E\cup F\in \cH$. (Note that the inverse image of $E\cup F$ in the $K3$ cover $S$ is the curve considered in \cite[Proposition 3.2, p. 349]{MoriMukai83}.) 

Therefore we will be done if we show that the closure of $\theta_{null}'$ contains such stable curves $E\cup F$. In fact all such curves belong to the closure of $\theta_{null}'$. For this we show that such curves have limit theta-characteristics (see \cite{Cornalba89}) with two independent global sections. This is sufficient because
\begin{itemize}
\item all the components of the set of stable curves with limit theta-characteristics with $\geq 2$ global sections have codimension $1$, 
\item the locus of curves $E\cup F$ is contained in exactly one boundary component of $\ocM_6$: the component $\Delta_0$ whose generic points parametrize irreducible curves, 
\item curves parametrized by generic points of $\Delta_0$ do not carry vanishing theta-characersitics.
\end{itemize}

Given a curve $E\cup F$ as above, suppose the points $p_1, p_2, p_3$ of $E$ are identified with the points $q_1, q_2, q_3$ of $F$ respectively. Consider the quasi-stable curve $E\cup R_1 \cup R_2 \cup R_3 \cup F$ where the smooth rational curve $R_i$ joins $p_i$ to $q_i$ and no other intersections exist between the components. Choose an odd theta-characteristic $\eta$ on $F$. Then a simple computation shows that the sheaf whose restriction to $E$ is $\cO_E$, whose restriction to $R_i$ is $\cO_{R_i}(1)$, and whose restriction to $F$ is $\eta$, is a vanishing theta-characteristic on $E\cup R_1 \cup R_2 \cup R_3 \cup F$. By Cornalba's description of the compactification of the moduli space of curves with theta-characteristics in \cite{Cornalba89}, this implies that $E\cup F$ belongs to the closure of $\theta_{null}'$.
\end{proof}

\begin{lemma}\label{lemmult1}
For $l$ and $T$ generic, the image of $l$ in $\cA_5$ meets $N_0$ transversely everywhere.
\end{lemma}

\begin{proof}
We need to prove that the image of $l$ in $\cA_5$ is not tangent to $N_0 = N'_0 \cup \theta_{null}$. Consider the image of $l$ in $\cR_6$. By \cite{SmithVarley1985-A5}, the branch divisor of the Prym map $P :\cR_6 \ra \cA_5$ is $N'_0$. The inverse image of $N'_0$ in $\cR_6$ is the union of the ramification divisor $R$ and the anti-ramification divisor $R'$. So $P^{-1} N_0 = R \cup R' \cup P^{-1} \theta_{null}$.

\begin{enumerate}

\item\label{lcapR} {\bf $l\cap R=\emptyset$.}\\
We first show that the image of $l$ in $\cR_6$ does not intersect $R$ and describe the singular points of the theta divisors at the points of $l\cap R'$ and $l\cap P^{-1} \theta_{null}$.\\
Let $t$ be a point of $l$ whose image lies in $N_0$ and let $P_t$ be the Prym variety of the cover $f_t := f |_{D_t} : D_t \ra C_t$. By \cite[pp. 342-343]{Mumford-PrymI-1974}, a singular point of the theta divisor $\T_t$ of $P_t$ corresponds to an invertible sheaf $M$ of canonical norm on $D_t$ such that, either $h^0 (M) \geq 4$, or $h^0(M) =2$ and there exists an invertible sheaf $N$ on $C_t$ with $h^0 (M\otimes f_t^*N^{-1}) >0$.\\
By \cite[Theorem 6.5]{FarkasGrushevskySalvatiManniVerra} we have $h^0 (M) \geq 4$ if and only if the double cover $D_t \ra C_t$ belongs to $R$.\\
Using the formula in \cite[Corollary 7.3]{FarkasGrushevskySalvatiManniVerra} for the divisor class of $R$ we compute that the degree of $R$ on the image of $l$ is $0$. Since $l$ is generic, it does not lie in $R$ hence it does not intersect $R$.\\
It also follows from the above argument that the Prym map is everywhere of maximal rank on the image of $l$ in $\cR_6$. In particular, by \cite{DonagiSmith81}, the curve $C_t$ is not hyperelliptic or trigonal.\\
We therefore have $M = f_t^* N (B)$ for an effective divisor $B$ on $D_t$ and a line bundle $N$ on $C_t$ of degree $4$ or $5$ such that $h^0 (N) =2$. By \cite[Proposition 7.1]{FarkasGrushevskySalvatiManniVerra}, when $N$ has degree $4$, the cover $D_t \ra C_t$ belongs to the antiramification divisor $R'$. When $N$ has degree $5$, $M=f_t^* N$ is a singular point of order $2$ on $\T_t$, hence the cover $D_t \ra C_t$ belongs to the inverse image of $\theta_{null}$ in $\cR_6$.

\item\label{qt} {\bf The equation $q_t$ of the tangent space to $N_0$.}\\
Let $\alpha$ be the point of order $2$ associated to the double cover $D_t \ra C_t$. In other words, $\alpha$ is the restriction of the canonical sheaf $K_T$ of $T$ to $C_t$. Since $h^0 (M) =2$, we have $h^0 (N \otimes \alpha) =0$. Furthermore, since $M$ has canonical norm, we have $N^2 (\oB) \cong \omega_{C_t}$ where $\oB := {f_t}_* B$.

Let $q_t \in S^2 H^1(\cO_{P_t})^*$ be an equation for the quadric tangent cone to the theta divisor $\T_t$ of $P_t$ at the singular point $M$ or $K_{D_t}\otimes M^{-1}$. Then, using the heat equation, it is easily seen, see, e.g., \cite[p. 87]{Mumford1975}, that under the identification $T_t \cA_5 \cong S^2 H^1(\cO_{P_t})$, $q_t$ is also an equation for the tangent space to $N_0$ at $t$.

Identifying the cotangent space to $P_t$ with the space $H^0 (K_{C_t} \otimes \alpha)$, the codifferential of the Prym map is identified with the multiplication map
\[
T_{P_t}^* \cA_5 \cong S^2 H^0 (K_{C_t}\otimes \alpha) \lra H^0 (K^2_{C_t}) \cong T_{C_t}^* \cM_6 
\]
(see \cite[p. 178]{Beauville1977-1}) where we identify the cotangent space to the moduli stack $\cR_6$ with that of the moduli stack $\cM_6$ via the natural projection. Since this map is an isomorphism, the quadric $q_t$ is determined by its zeros on the Prym-canonical image of $C_t$ in $\bP H^0 (K_{C_t} \otimes \alpha)^* = \bP H^1 (\cO_{P_t})$.

\item {\bf The divisor of $q_t$ on $C_t$.}\\
First note that $M$ is a double point of rank $3$ on the theta divisor $\T_{D_t}$ of $JD_t$: the rulings of the tangent cone $\gamma_t$ to $\T_{D_t}$ cut the moving parts of $M$ and $K_{D_t}\otimes M^{-1}$ on $D_t$. By part \eqref{lcapR} above, $M \cong f_t^* N (B)$ and $f_t^* N^2 (B+\sigma B) \cong f_t^* K_{C_t} \cong K_{D_t}$, so the moving parts of $M$ and $K_{D_t}\otimes M^{-1}$ and hence the two rulings of $\gamma_t$ are equal. Therefore, by, e.g., \cite[p. 11]{SmithVarley2006-Pfaffian}, locally around $M$, an equation $s$ of the theta divisor of $JD_t$ has an expansion of the form
\[
s = x^2 - yz + Q^2 + \hbox{higher order terms}
\]
where $x,y,z$ are suitable analytic coordinates centered at $M$ and $Q$ is homogeneous of degree $2$. By \cite[p. 343]{Mumford-PrymI-1974}, since $h^0(M)=h^0(N)=2$, the second degree term $x^2 - yz$ vanishes on the tangent space to $P_t$. Therefore $q_t$ is the restriction of $Q$ to $T_0 P_t = H^1 (\cO_{P_t})$. By \cite[p. 349, end of $\S 1$]{KempfSchreyer1988}, the trace of $Q^2$ on the canonical image of $D_t$ is the divisor
\[
2R_{f_t^* N} + 2B + 2\tau B
\]
where $R_{f_t^* N} = f_t^* R_N$ is the ramification divisor of the pencil $|f_t^* N|$. Therefore the trace of $Q$ on $D_t$ is $f_t^* R_{N} + B + \tau B$. For a point $p\in C_t$ with inverse images $p'$ and $p''$ in $D_t$, the Prym-canonical image of $p$ is the intersection of the span $\langle p' + p'' \rangle \subset \bP H^0 (K_{D_t} )^*$ with the $\tau$-anti-invariant subspace $\bP H^0 (K_{C_t} \otimes \alpha)^*$. If $u$ and $v$ are homogeneous coordinates on $\langle p' + p'' \rangle$ with respective zeros $p'$ and $p''$, then $\tau^* u =v$, and $u-v$, $u+v$ are, respectively, the $\tau$-anti-invariant and $\tau$-invariant coordinates on $\langle p' + p'' \rangle$. Assume that either $p$ is a ramification point of $|N|$ or, if $B$ is nonzero, a point of the support of $\oB := {f_t}_* B$. Since $Q$ contains $p'$ and $p''$, it restricts to a multiple of $uv = \frac{1}{4} ((u+v)^2 - (u-v)^2)$ on $\langle p' + p'' \rangle$. The restriction of $Q$ to $\bP H^0 (K_{C_t} \otimes \alpha)^*$ is obtained by setting its $\tau$-invariant coordinates to $0$. Therefore the restriction of $Q$ to $\bP H^0 (K_{C_t} \otimes \alpha)^*$ vanishes on the Prym-canonical image of $p$ whose equation on $\langle p' + p'' \rangle$ is $u-v$.

Therefore the divisor of zeros of $q_t$ on $C_t$ is $\frac{1}{2}{f_t}_* (f_t^* R_{N} + B + \tau B) = R_N + \oB \in |K_{C_t}^2|$.

\item {\bf Transversality in terms of the cup product of $q_t$ with the extension class of the tangent bundle sequence.}\\
Consider the tangent bundle sequence
\[
0 \lra \cT_{C_t} \lra \cT_T |_{C_t} \lra \cO_{C_t} (C_t) \lra 0.
\]
The connecting homomorphism
\[
H^0 (\cO_{C_t} (C_t)) \lra H^1 (\cT_{C_t}) = H^0 (K_{C_t}^2)^*
\]
is the Kodaira-Spencer map of the family of curves parametrized by $|\cO_T (C_t)|$. It is given by cup-product with the extension class $\epsilon \in H^1 (\cT_{C_t} (-C_t))$ of the tangent bundle sequence. To show that the image of a generic line $l$ is not tangent to $N_0$, we need to show that the hyperplane defined by $q_t$ in $H^1 (\cT_{C_t})$ does not contain the image of $H^0 (\cO_{C_t} (C_t))$. In other words, we need to show that $q_t \ {}_\cup \ \epsilon \in H^0 (\cO_{C_t}(C_t))^* = H^1 (\alpha)$ is not zero.

\item {\bf Transversality via Lazarsfeld-Mukai extensions.}\\
Let $b$ and $r_N$ be sections of $\cO_{C_t} (B)$ and $K_{C_t}\otimes N^2$ with divisors of zeros $\oB$ and $R_N$ respectively, so that $q_t = b \ {}_\cup \ r_N$. The Lazarsfeld-Mukai bundle $F$ on $T$ defined by the natural exact sequence
\begin{equation}\label{eqF}
0 \lra F^* \lra H^0 (N) \otimes \cO_T \lra N \lra 0.
\end{equation}
Note that $N$ is base point free. When $N$ has degree $4$, this is a consequence of the fact from part \eqref{lcapR} that $C_t$ is not trigonal. When $N$ has degree $5$, $N$ is a vanishing theta-null on $C_t$. The double cover $D_t \ra C_t$ is, by Lemma \ref{lemgenthetanull}, parametrized by a generic point of $\theta_{null}'$, hence $N$ is base point free also in this case. \\
Put $E := F |_{C_t}$. As in \cite[p. 252]{Voisin1992-Wahl} we have the commutative diagram with exact rows
\[
\xymatrix{
0\ar[r]&K_{C_t}^{-1}\otimes N\otimes\alpha\ar[r]\ar@{=}[d]&E^* \ar[r]\ar[d]&N^{-1}\ar[r]\ar[d]&0\\
0\ar[r]&K_{C_t}^{-1}\otimes N\otimes\alpha\ar[r]&\Omega^1_T |_{C_t}\otimes N\ar[r]&K_{C_t} \otimes N\ar[r]&0}
\]
which shows that $r_N \ {}_\cup \ \epsilon$ is the extension class for the extension
\begin{equation}\label{eqE}
0 \lra K_{C_t}^{-1} \otimes N \otimes \alpha \lra E^* \lra N^{-1} \lra 0.
\end{equation}
Using the fact that $M$ has canonical norm, we obtain $N^2 (\oB) \cong K_{C_t}$ so that $K_{C_t}^{-1} \otimes N \otimes \alpha \cong N^{-1} (-\oB) \otimes \alpha$. Pushing forward sequence \eqref{eqE} via multiplication by $b$ : $N^{-1} (-\oB)\otimes \alpha \ra N^{-1} \otimes \alpha$, we obtain that $b \ {}_\cup \ r_N \ {}_\cup \ \epsilon$ is the extension class for the  pushforward extension
\begin{equation}\label{eqG}
0 \lra N^{-1} \otimes \alpha \lra G^* \lra N^{-1} \lra 0.
\end{equation}
So, to complete the proof of the lemma, we need to prove that this extension is not split. 

\item {\bf The Lazarsfeld-Mukai extension does not split.}\\
Dualizing sequence \eqref{eqF}, we obtain the exact sequence
\begin{equation}\label{eqF*}
0 \lra H^0 (N)^* \otimes \cO_T \lra F \lra N^{-1} \otimes \cO_{C_t} (C_t) \lra 0.
\end{equation}
or
\begin{equation}\label{eqF**}
0 \lra H^0 (N)^* \otimes \cO_T \lra F \lra N (\oB) \otimes \alpha \lra 0.
\end{equation}
Twisting \eqref{eqF**} by $F^*$ we obtain
\begin{equation}\label{eqFF*}
0 \lra H^0 (N)^* \otimes F^* \lra F \otimes F^* \lra N (\oB) \otimes \alpha \otimes F^* \lra 0.
\end{equation}
From the cohomology of sequence \eqref{eqF} we obtain $h^0 (F^*) = h^1(F^*) =0$. Therefore, the cohomology of sequence \eqref{eqFF*} gives the isomorphism $H^0 (F \otimes F^*) = H^0 (N (\oB) \otimes \alpha \otimes F^*)$. Twisting sequence \eqref{eqE} with $N (\oB) \otimes \alpha$ and taking cohomology we obtain $h^0 (N (\oB) \otimes \alpha \otimes E^*) = 1$ because $h^0 (\alpha (\oB)) =0$ since $h^0 (B + \tau B) =1$. Now $N (\oB) \otimes \alpha \otimes E^*\cong N (\oB) \otimes \alpha \otimes F^*$ and we deduce $h^0 (F \otimes F^*) =1$.

Twisting \eqref{eqF} with $K_T$ and taking cohomology we obtain $h^0 (F^*\otimes K_T) = h^1(F^*\otimes K_T) =0$. Tensoring sequence \eqref{eqFF*} with $K_T$ and taking cohomology we then obtain the isomorphism $H^0 (F \otimes F^* \otimes K_T) = H^0 (N (\oB) \otimes F^*)$.

Assume now that sequence \eqref{eqG} splits. Then there exists a surjective map $G \ra N$, which implies $H^0 (G^* \otimes N) \neq 0$. Sequences \eqref{eqE} and \eqref{eqG}, after tensoring with $N(\oB)$, are part of the commutative diagram with exact rows and columns
\[
\xymatrix{& 0\ar[d] & 0\ar[d] &  \\
0\ar[r]&\alpha\ar[r]\ar[d]&E^* \otimes N (\oB)\ar[r]\ar[d]&\cO_{C_t} (\oB)\ar[r]\ar[d]^{\cong}&0\\
0\ar[r]&\alpha (\oB)\ar[d]\ar[r]&G^* \otimes N (\oB)\ar[d]\ar[r]&\cO_{C_t} (\oB)\ar[r]&0\\
&\alpha (\oB) |_{\oB}\ar[r]^{\cong}\ar[d]&\cO_{\oB}\ar[d]&&\\
 & 0 & 0. &}
\]
Since the sections of $G^* \otimes N$ can be interpreted as the sections of $G^*\otimes N(\oB)$ that vanish on $\oB$, it follows from the above diagram that we have natural inclusions
\[
H^0 (G^* \otimes N) \inj H^0 (E^* \otimes N (\oB)) \inj H^0 (G^* \otimes N (\oB)).
\]
Therefore, if $H^0 (G^* \otimes N) \neq 0$, we also have $H^0 (N (\oB) \otimes F^*) = H^0 (E^* \otimes N (\oB)) \neq 0$.

Summarizing, if sequence \eqref{eqG} splits, then $H^0 (F^* \otimes F \otimes K_T) \neq 0$. So there exists a nonzero homomorphism
\[
\varphi : F \lra F \otimes K_T.
\]
Since $h^0 (F \otimes F^*) =1$, the composition $(\varphi \otimes K_T) \circ \varphi$ is a multiple of the identity. Furthermore, $(\varphi \otimes K_T) \circ \varphi$ cannot be an isomorphism because otherwise $\varphi$ would have maximal rank everywhere hence would also be an isomorphism. \\
Note that $F$ and $F\otimes K_T$ are not isomorphic: any isomorphism would restrict to an isomorphism $H^0 (N)^* \otimes \cO_{T \setminus C_t} \ra H^0 (N)^* \otimes K_T |_{T \setminus C_t}$ which, after composing with an embedding $\cO_{T \setminus C_t} \inj H^0 (N)^* \otimes \cO_{T \setminus C_t}$ and a projection $H^0 (N)^* \otimes K_T |_{T \setminus C_t} \surj K_T |_{T \setminus C_t}$, would give an embedding $\cO_{T \setminus C_t} \subset K_T |_{T \setminus C_t}$. This would in turn imply that $K_T|_{T \setminus C_t}$ is trivial which is not possible since it would mean the cover $S\setminus D_t \ra T \setminus C_t$ is trivial. \\
Therefore $(\varphi \otimes K_T) \circ \varphi =0$. Similarly, $\varphi$ cannot have maximal rank anywhere since otherwise the same would be true of $\varphi \otimes K_T$ and of $(\varphi \otimes K_T) \circ \varphi$. Therefore the kernel of $\varphi$ is a subsheaf of rank $1$ of $F$ and the image of $\varphi$ is a subsheaf of rank $1$ of $F\otimes K_T$.

Next note that the isomorphism
\[
\begin{array}{rcl}
H^0 (F^* \otimes F \otimes K_T) & \stackrel{\cong}{\lra} & H^0 (F^* \otimes N(\oB)) = H^0 (E^* \otimes N(\oB))\\
\varphi & \longmapsto & \overline{\varphi}
\end{array}
\]
above is given by composing a homomorphism $\varphi : F \ra F \otimes K_T$ with the surjection $F \otimes K_T \surj N (\oB)$ appearing in sequence \eqref{eqF**} after twisting with $K_T$. Since the image of $\overline{\varphi}\in H^0 (E^* \otimes N (\oB))$ in $H^0 (\cO_{C_t}(\oB))$ by the map $H^0 (E^* \otimes N (\oB)) \ra H^0(N^{-1} \otimes N (\oB)) = H^0 (\cO_{C_t} (\oB))$ obtained from sequence \eqref{eqE} is nonzero, the composition
\[
N \lra E \stackrel{\overline{\varphi}}{\lra} N (\oB)
\]
is nonzero, hence injective. It follows that the image of $\overline{\varphi}$ contains the subsheaf $N$ of $N(\oB)$.

Define the torsion free sheaf $F'$ on $T$ as the kernel of the composition $F \surj E \surj \cO_{\oB}$ where the map $E \surj \cO_{\oB}$ is the composition $E \surj K_{C_t} \otimes N^{-1} \otimes \alpha \cong N (\oB) \otimes \alpha \surj (N (\oB) \otimes \alpha) |_{\oB} \cong \cO_{\oB}$. Note that by definition we have the exact sequence
\begin{equation}\label{eqFF'}
0 \lra F' \lra F \lra \cO_{\oB} \lra 0.
\end{equation}
Since $\overline{\varphi} \in H^0 (G^* \otimes N)$, a moment of reflection will show that the composition $F' \inj F \stackrel{\varphi}{\ra} F\otimes K_T$ factors through a homomorphism $\psi : F' \ra F' \otimes K_T$ whose composition with $F' \otimes K_T \ra F \otimes K_T \surj N (\oB)$ factors through $N \inj N (\oB)$. So we have the nonzero composition
\[
\overline{\psi} : F' \stackrel{\psi}{\lra} F' \otimes K_T \surj N
\]
which factors through $\overline{\varphi} :G \ra N$. Since the composition $N \ra G \stackrel{\overline{\varphi}}{\ra} N$ is induced by $N \ra E \stackrel{\overline{\varphi}}{\ra} N (\oB)$, we obtain that $\overline{\psi}$ is surjective. Therefore the image sheaf $\Img (\psi)$ is a torsion free rank $1$ sheaf on $T$ whose restriction to $C_t$ is $N$. Let $X$ be a divisor on $T$ representing $c_1 (\Img (\psi))$. Then, by, e.g., \cite[pp. 339-350]{BHPV}, $X$ is effective of non-negative self-intersection because $X \cdot C_t$ is positive and $T$ is generic. Furthermore $Y := C_t - 2X$ is also effective since its restriction to $C_t$ is $\oB$ which is effective. Since $h^0 (T, Y) \leq h^0 (C_t, Y |_{C_t}) = h^0 (\oB) =1$, we have $h^0 (Y) \leq 1$ which implies $Y$ has arithmetic genus $1$ (since $T$ is generic and does not contain curves of arithmetic genus $0$). Therefore $Y^2 =0$. We have the linear equivalence of effective divisors $C_t \equiv 2X +Y$. Hence $10 = C_t^2 = 4 X^2 + 4 X \cdot Y$ is a multiple of $4$ which is not possible.
\end{enumerate}
\end{proof}

We can now prove

\begin{proposition}\label{proptotsmooth}
The total spaces $A$ and $\T$ are smooth.
\end{proposition}

\begin{proof}
We show that the tangent spaces to $A$ and $\T$ have dimension $6$ and $5$ respectively everywhere. Let $p\in A_t$, resp. $p\in \T_t$, be a point of the fiber of $A \ra l$, resp. $\T \ra l$, at $t\in l$. If $A_t$ is smooth at $p$, it follows from \cite[Proposition 3.1]{IzadiTamasWang} and Lemma \ref{lemmult1} that, for a generic choice of $l$, both $A$ and $\T$ (when $p\in\T$) are smooth at $p$. Assume therefore that $A_t$ is singular at $p$. In such a case, it follows from the description of $\T_t$ in Proposition \ref{propdegtheta} that, if $p\in \T$, $\T_t$ is also singular at $p$. By the description of $A_t$ in Section \ref{subseccompPrym}, resp. $\T_t$ in Proposition \ref{propdegtheta}, the tangent space to $A_t$ at $p$, resp. $\T_t$ at $p$, has dimension $6$, resp. $5$. We therefore need to show that the tangent space to the total space $A$, resp. $\T$, is equal to the tangent space of the fiber. The tangent space to the fiber is the kernel of the differential of the map $A \ra l$, resp. $\T \ra l$. Since the map $\T \ra l$ is the scheme-theoretic restriction of the map $A\ra l$, we need to show that the differential of the map $A\ra l$ is $0$ at $p$ to obtain the smoothness of $A$ at $p$ and also of $\T$ at $p$ when $p\in\T$.

The total space $A$ is the inverse image of the generic line $l \subset |H|$ in the relative Prym variety $P_H \ra |H|$ constructed in \cite{ArbarelloFerrettiSacca}. By \cite[Prop. 3.10, Prop. 4.4, Prop. 5.1]{ArbarelloFerrettiSacca}, the singular locus of $P_H$ lies above a union of lines or points $m_i$ in $|H|$. We can therefore assume that $l$ does not meet any of the $m_i$. Furthermore, since all pull-backs are scheme-theoretic and all fibers reduced, the restriction of the differential of $P_H \ra |H|$ to $A$ is the differential of the projection $A \ra l$. The rank of the differential of $P_H \ra |H|$ is not maximal at $p$ (see loc. cit.), i.e., its image is a proper subspace of the tangent space of $|H|$ at $t$. Since $l$ is generic, the tangent space of $l$ at $t$ intersects this image in $0$. Therefore the differential of $A \ra l$ is $0$ at $p$.
\end{proof}

To summarize, our family of (compactified) Prym varieties and theta divisors
\[
\T\lra A\lra l
\]
has smooth total spaces, $240$ fibers where the theta divisor has a single node, $60$ fibers where theta has two nodes, and $42$ fibers where theta is as in Proposition \ref{propdegtheta}.

\section{Local monodromy representations near $N_0$} \label{secN0monodromy}
\subsection{Local monodromy near $\theta_{null}$}\label{subsecnull}
The local monodromy representation on the cohomology of the theta divisor near a general point $(A_0,\T_0)\in\theta_{null}$ is given by the classic Picard-Lefschetz formula. Fix a point $p_0\in l\cap\theta_{null}$ and pick a small disk $U\subset l$ containing $p_0$. We have a family of theta divisors with smooth total space $\T_U$ (see Proposition \ref{proptotsmooth}):
\[
\xymatrix{\T_0\ar[r]\ar[d]&\T_U\ar[d]\\
p_0\ar[r]&U.}
\]
The local monodromy representation on the cohomology of a general fiber $\T_t$ for $t\in U\setminus\{p_0\}$
\[
\rho:\pi_1(U\setminus\{p_0\},t)\longrightarrow GL(H^k(\T_t))
\]
is trivial when $k\ne4$. When $k=4$, the Picard-Lefschetz formula (see, for instance, \cite[p. 93]{Landman1973}) shows that $\rho(\pi_1(U\setminus\{p_0\},t))$ is generated by 
\[
\begin{array}{ccc}
T_U: H^4(\T_t) & \lra & H^4(\T_t) \\
\alpha & \mapsto & \alpha-\langle\alpha,\gamma\rangle\gamma
\end{array}
\]
where $\langle,
\rangle$ is the intersection product on $H^4(\T_t)$, and $\gamma\in H^4(\T_t)$ is the class of the vanishing $4$-sphere with $\langle\gamma,\gamma\rangle=2$.

One checks immediately that 
\begin{eqnarray}\label{eqnLP}
T_U^2=Id. 
\end{eqnarray}

\subsection{Local monodromy near $N_0'$}\label{subsectwonodes}
Fix a point $p_0\in l\cap N_0'$ and a small disk $U\subset l$ containing $p_0$. The central fiber $\T_0$ of the family $\T_U$ has two ordinary double points $x$ and $-x$. The local analysis in \cite[pp. 99-100, 103-106]{Landman1973} applies and we still have
\begin{eqnarray}\label{eqnmono}
T_U^2=Id. 
\end{eqnarray}

\section{General facts about the Clemens-Schmid exact sequence}\label{secClemensSchmid}
We briefly review some general facts about the Clemens-Schmid exact sequence. We will apply the general theory in this section to compute the local monodromy representations near the boundary $\Delta$.

\subsection{The Clemens-Schmid exact sequence}
Let
\[\xymatrix{Y_0\ar[r]\ar[d]&\cY\ar[d]&Y_t\ar[l]_-{i_t}\ar[d]\\
\{0\}\ar[r]&V&\{t\}\ar[l]}
\] 
be a one-parameter semistable degeneration (i.e., the total space $\cY$ is smooth and the central fiber $Y_0$ is reduced with simple normal crossing support) over a small disk $V$, and $0\ne t\in \partial V$ a general point. The total space $\cY$ deformation retracts to $Y_0$.
For such a family, the image of the monodromy representation
\[
\rho: \pi_1(V\setminus\{0\},t)\lra GL(H^\bullet(Y_t))
\]
is generated by a unipotent operator $T:H^\bullet(Y_t)\ra H^\bullet(Y_t)$, i.e. $(T-Id)^k=0$ for some integer $k$ \cite{Landman1973}. Thus 
\[
N:=\log T:=(T-Id)-\frac{1}{2}(T-Id)^2+\frac{1}{3}(T-Id)^3+...\] 
is nilpotent.

It follows from the work of Clemens-Schmid \cite{clemens77}, \cite{schmid73} and
Steenbrink \cite{Steenbrink1975} that one can define mixed Hodge structures on $H^\bullet(Y_t)$,
$H^\bullet(\cY)$ and $H_\bullet(\cY)$ such that we have an exact sequence of mixed Hodge structures (with suitable weight shifts)
\begin{eqnarray}\label{eqCS}
\xymatrix{\ar[r]&H_{2n+2-m}(\cY)\ar[r]^-{\alpha}&H^m(\cY)\ar[r]^-{i_t^*}&H^m(Y_t)_{\lim}\ar[r]^-{N}&H^m(Y_t)_{\lim}\ar[r]^-{\beta}&
  H_{2n-m}(\cY)\ar[r]&}
\end{eqnarray}
where $n$ is the relative dimension of the fibration, $\alpha$ is the composition
\begin{eqnarray}
\xymatrix{H_{2n+2-m}(\cY)\ar[r]^-{\text{ PD}}&H^m(\cY,\partial\cY)\ar[r]&H^m(\cY),}
\end{eqnarray}
and $\beta$ is the composition
\begin{eqnarray}\label{eqnbeta}
\xymatrix{H^m(Y_t)\ar[r]^-{\text{PD}}&H_{2n-m}(Y_t)\ar[r]^-{i_{t*}}&H_{2n-m}(\cY).}
\end{eqnarray}
Here `PD' stands for Poincar\'e duality. The mixed Hodge structure on $H^\bullet(Y_t)$ is not the usual pure Hodge structure but rather the `limit mixed Hodge structure' (c.f. Section \ref{subseclim}). We use the notation $H^\bullet(Y_t)_{\lim}$ to distinguish it from the pure Hodge structure.
\subsection{The weight filtrations on $H^m(\cY)$ and $H_m(\cY)$}\label{subsecspseq}   Put 
\begin{eqnarray}
 &&H^m :=H^m(\cY)\cong H^m(Y_0),\nonumber\\
 &&H_m :=H_m(\cY)\cong H_m(Y_0).\nonumber
\end{eqnarray}
Recall from \cite[p. 103]{morrison84} that there is a Mayer-Vietoris type spectral sequence abutting to $H^{\bullet}
(Y_0)$ with $E_1$ term 
\[
E_1^{p,q} = H^q (Y_0^{[p]}).
\]
Here $Y_0^{[p]}$ is the disjoint union of the codimension $p$ strata of $Y_0$, i.e.,
\[
Y_0^{[p]} := \coprod_{i_0 < \ldots < i_p} Z_{i_0}\cap \ldots \cap Z_{i_p}
\]
where the $Z_i$ are distinct irreducible components of $Y_0$. 
 
The differential $d_1$
\[
\xymatrix{E_1^{p,q}\ar[d]^-{\cong}\ar[r]^-{d_1}&E_1^{p+1,q}\ar[d]^-{\cong}\\
H^q(Y_0^{[p]})\ar[r]^-{d_1}&H^q(Y_0^{[p+1]})}\]
is the alternating sum of the restriction maps on all the irreducible components.
By \cite[p. 103]{morrison84} this sequence degenerates at $E_2$. 

The weight filtration is given by
\[
W_kH^m := \oplus_{p+q=m,\ q\le k} E_{\infty}^{p,q} = \oplus_{p+q=m,\ q\le k} E_{2}^{p,q}.
\]
Therefore the weights on $H^m$ go from $0$ to $m$ and 
\[
Gr_kH^m\cong E_2^{m-k,k}=\frac{\Ker(d_1: H^{k}(Y_0^{[m-k]})\ra H^{k}(Y_0^{[m-k+1]})}{\Img(d_1: H^{k}(Y_0^{[m-k-1]})\ra H^{k}(Y_0^{[m-k]})}.
\]
There is also a weight filtration on $H_m$:
\[
W_{-k}H_m:=(W_{k-1}H^m)^{\perp}
\]
under the perfect pairing between $H^m$ and $H_m$. With this definition,
\[
Gr_{-k}H_m\cong (Gr_k H^m)^{\lor}.
\]

\subsection{The limit mixed Hodge structure $H^m(Y_t)_{\lim}$}\label{subseclim} 
The weight filtration associated to the nilpotent operator $N$ has the following form,
\[
0\subset W_0\subset W_1\subset...\subset W_{2m}=H^m(Y_t).
\]
We refer to
\cite[pp. 106-109]{morrison84} for the precise definition of the monodromy weight filtration and only summarize the properties we need here.

In the applications in this paper, the nilpotent operator $N$ satisfies
\[
N^2=0.
\]
Thus the monodromy weight filtration satisfies the following
\begin{eqnarray}
&& W_k=0  \ \hbox{ for } \ k\le m-2,\nonumber\\
&& W_{m-1}=\Img(N), \nonumber\\
 && W_m=\Ker(N),\nonumber\\
 && W_k= H^m(Y_t) \ \text{ for } \ k\ge m+1.\nonumber
\end{eqnarray}

Let $K^m_t :=\Ker (N)\subset H^m(Y_t)$ be the monodromy invariant subspace. It inherits an induced weight filtration from $H^m(Y_t)$. 
The graded pieces of $H^m(Y_t)_{\lim}$ thus satisfy
\begin{eqnarray}
&&\label{eqnKHlim}Gr_{m}H^m(Y_t)_{\lim}\cong Gr_mK^m_t\cong\frac{\Ker(N)}{\Img(N)}\\
&&\label{eqnsym}Gr_{m+1}H^m(Y_t)_{\lim}\stackrel{N}\cong Gr_{m-1}H^m(Y_t)_{\lim}\cong Gr_{m-1}K^m_t\cong \Img(N).
\end{eqnarray}

The weight filtrations on $H^m$ and $K_t^m$ are related by the Clemens-Schmid exact sequence. Below are the basic facts we will use (see \cite[pp. 107-109]{morrison84})

\begin{enumerate}
\item $i_t^*$ induces an isomorphism
\begin{eqnarray}\label{Grsmallk}\xymatrix{Gr_kH^m\ar[r]^-{\cong}& Gr_kK^m_t}\ \hbox{ for }\ k\le m-1.
\end{eqnarray}
 \item There is an exact sequence
\begin{eqnarray}\label{seqGR4}\xymatrix{0\ar[r]&Gr_{m-2}K^{m-2}_t\ar[r]&Gr_{m-2n-2}H_{2n+2-m}\ar[r]^-{\alpha}&Gr_mH^m\ar[r]&Gr_mK^m_t\ar[r]&0}.
\end{eqnarray}
\end{enumerate}

The limit Hodge filtration on $H^m(Y_t)_{\lim}$ is given by (\cite{morrison84}, \cite{schmid73})
\begin{eqnarray}\label{eqnlimit}
F_\infty^p=\lim_{\im z\to\infty}\exp(-zN) F^p(z)
\end{eqnarray}
where $f:U'\ra U\setminus\{0\}$, $f(z)=e^{2\pi iz}$ is the universal cover of the punctured disk and $F^p$ is the usual Hodge filtration on $H^m(Y_{f(z)})$ on the fixed underlying space $H^m(Y_t)$.

\section{Local monodromy near the boundary $\Delta$}\label{secboundary} Near the boundary $\Delta$, the family of Prym varieties $A_U\ra U$ parametrized by a small disk $U\subset l$ has smooth general fiber $(A_t,\T_t)$ and central fiber $({P},\Upsilon)$ as in Proposition \ref{propdegtheta}. We use the Clemens-Schmid exact sequence to compute the monodromy action.

\subsection{The semi-stable reduction}\label{subsecbdy}
For the semi-stable reduction of any family, one usually blows up the non-normal crossings loci of the central fiber. As the total space $A_U$ is smooth, to avoid multiple components in the central fiber, we first make a ramified base change $V\ra U$ of order $2$ of the family
\[ 
\xymatrix{A_V\ar[r]\ar[d]&A_U\ar[d]\\
V\ar[r]&U.}
\] 
We then blow up the singular locus ${P}\setminus G_{pq}$ of $A_V$ to obtain the family
$\widetilde{A}_V\ra V$.

\begin{proposition}\label{propA0}The central fiber $\widetilde{A}_0$ of the family $\widetilde{A}_V \ra V$ is the union of two copies $P_1^\nu$ and $P_2^\nu$ of $P^\nu$, with $B_0\subset P^\nu_1$ identified with $B_\infty\subset P^\nu_2$ via the identity map and $B_\infty\subset P^\nu_1$  identified with $B_0\subset P^\nu_2$ via translation by $b$. The intersection $P_1^\nu\cap P_2^\nu=B_{0\infty}\amalg B_{\infty0}$ is the disjoint union of two copies of $B$. \end{proposition}
\begin{proof} Clearly the main component $P_1^\nu\cong P^\nu$. We will show the exceptional divisor $P_2^\nu$ is also isomorphic to $P^\nu$. In the semistable family $\widetilde{A}_V\ra V$, we have 
\[
N^\lor_{B_{0\infty}/P_1^\nu}\cong N_{B_{0\infty}/P_2^\nu}.
\]
Therefore $P_2^\nu$ contains the total space of $\cO_B(\Xi_b-\Xi)\cong\cO_{B_0}(-B_0)\cong N_{B_{0\infty}|P_2^\nu}=P^\nu_2\setminus B_{\infty0}$ as a Zariski open subset. Applying the same argument to $B_{\infty0}$, we see that $P_2^\nu$ also contains the total space of $\cO_B(\Xi-\Xi_b)\cong N_{B_{\infty0}|P^\nu_2}=P^\nu_2\setminus B_{0\infty}$ as an open subset. We conclude that $P_2^\nu\cong\bP_B(\cO_B(\Xi-\Xi_b)\oplus\cO_B(\Xi_b-\Xi))\cong P^\nu$. The statement about the gluing follows from the fact that after contracting $P^\nu_2$, the infinity and zero sections of $P_1^\nu$ are identified via translation by $b$.
\end{proof}

\begin{corollary}\label{proptheta0}The central fiber $\widetilde{\T}_0$ of the family $\widetilde{\T}_V \ra V$ is the union $\Upsilon^\nu\cup Q_\Xi$, where $\Upsilon^\nu=Bl_{\Xi\cap\Xi_b}B$ and the conic bundle $Q_\Xi$ is the restriction of $P_2^\nu\ra B$ to $\Xi$. The intersection $\Upsilon^\nu\cap Q_\Xi=\Xi_{0\infty}\amalg \Xi_{\infty0}$ is the disjoint union of two copies of $\Xi$.
\end{corollary}

\begin{proof}Immediate.
\end{proof}

\subsection{The weight filtration on $H^m(\widetilde{A}_0)$}
By Section \ref{subsecspseq} and Proposition \ref{propA0}, the weight filtration on $H^m(\widetilde{A}_0)$ only has the following possibly nontrivial graded pieces
\[
Gr_mH^m(\widetilde{A}_0)=\Ker(d_1:H^m(P_1^\nu)\oplus H^m(P_2^\nu)\lra H^m(B_{0\infty})\oplus H^m(B_{\infty0}))
\]
and
\[
Gr_{m-1}H^m(\widetilde{A}_0)=\Coker (d_1:H^{m-1}(P_1^\nu)\oplus H^{m-1}(P_2^\nu)\lra H^{m-1}(B_{0\infty})\oplus H^{m-1}(B_{\infty0}))
\]
\begin{proposition}\label{propGrA}We have
\[
Gr_mH^m(\widetilde{A}_0)\cong H^{m-2}(B)\oplus H^m(P^\nu) \cong H^{m-2}(B)^{\oplus 2}\oplus H^m(B),
\]
and
\[
Gr_{m-1}H^m(\widetilde{A}_0)\cong H^{m-1}(B).
\]
\end{proposition}
\begin{proof}By Remark \ref{rktrivial}, $P^\nu\ra B$ is a topologically trivial $\bP^1$ bundle. The statements then follow easily from Proposition \ref{propA0} and the K$\ddot{\text{u}}$nneth formula. Note that, in the first formula, since the differential $d_1$ is the difference of the restriction maps from $P_1^{\nu}$ and $P_2^{\nu}$, its kernel $Gr_mH^m(\widetilde{A}_0)\cong H^{m-2}(B)\oplus H^m(P^\nu)$ is embedded diagonally in $H^m(P_1^\nu)\oplus H^m(P_2^\nu)\cong H^m(P^\nu)\oplus H^m(P^\nu)$.
\end{proof}

\begin{corollary}\label{corGrAt}The monodromy weight filtration on $H^m(A_t)_{\lim}$ satisfies
\[
Gr_{m+1}H^m({A}_t)_{\lim}\cong Gr_{m-1}H^m({A}_t)_{\lim}\cong H^{m-1}(B)
\]
and
\[
Gr_mH^m(A_t)_{\lim}\cong H^m (B) \oplus H^{m-2} (B).
\]
\end{corollary}
\begin{proof} For the first part, by (\ref{eqnsym}) and (\ref{Grsmallk}) we have $Gr_{m+1}H^m({A}_t)_{\lim}\cong Gr_{m-1}H^m({A}_t)_{\lim}\cong Gr_{m-1}H^m(\widetilde{A}_0)$ which is isomorphic to $H^{m-1}(B)$ by Proposition \ref{propGrA}. For the second part, use \eqref{eqnKHlim} and Proposition \ref{propGrA}, and apply sequence \eqref{seqGR4} repeatedly.
\end{proof}

\subsection{The weight filtration on $H^m(\widetilde{\T}_0)$}
By Section \ref{subsecspseq} and Proposition \ref{proptheta0}, the weight filtration on $H^m(\widetilde{\T}_0)$ only has the following possibly nontrivial graded pieces
\[
Gr_mH^m(\widetilde{\T}_0)=\Ker(d_1:H^m(\Upsilon^\nu)\oplus H^m(Q_\Xi)\lra H^m(\Xi_{0\infty})\oplus H^m(\Xi_{\infty0}))
\]
and
\[
Gr_{m-1}H^m(\widetilde{\T}_0)=\Coker (d_1:H^{m-1}(\Upsilon^\nu)\oplus H^{m-1}(Q_\Xi)\lra H^{m-1}(\Xi_{0\infty})\oplus H^{m-1}(\Xi_{\infty0}))
\]
\begin{proposition}\label{propGrtheta}
We have
\[
Gr_mH^m(\widetilde{\T}_0)\cong H^m(B)\oplus H^{m-2}(\Xi\cap\Xi_b)\oplus H^{m-2}(\Xi)   ,
\]
and
\[
Gr_{m-1}H^m(\widetilde{\T}_0)\cong H^{m-1}(\Xi).
\]
\end{proposition}
\begin{proof}By Corollary \ref{proptheta0}, $H^m(\Upsilon^\nu)\cong H^m(B)\oplus H^{m-2}(\Xi\cap\Xi_b)$ and the restriction map $H^m(\Upsilon^\nu)\ra H^m(\Xi_{0\infty})$ can be identified with the map
\[
H^m(B)\oplus H^{m-2}(\Xi\cap\Xi_b)\stackrel{(j^*, i_*)}\lra H^m(\Xi).
\]
Thus the image of 
\[
H^m(\Upsilon^\nu)\lra H^m(\Xi_{0\infty})\oplus H^m(\Xi_{\infty0})
\] 
 is contained in the image of 
\[
H^m(Q_\Xi)\cong H^m(\Xi)\oplus H^{m-2}(\Xi)\lra H^m(\Xi_{0\infty})\oplus H^m(\Xi_{\infty0}),
\]
which is equal to the diagonal of $H^m(\Xi_{0\infty})\oplus H^m(\Xi_{\infty0})$. Thus 
\[
Gr_{m-1}H^m(\widetilde{\T}_0)\cong H^{m-1}(\Xi).
\]
Next we compute $Gr_mH^m(\widetilde{\T}_0)\subset H^m(\Upsilon^\nu)\oplus H^m(Q_\Xi)$. By the previous discussion, for any $x\in  H^m(\Upsilon^\nu)$, we can find $y\in H^m(Q_\Xi)$ such that $(x,y)\in Gr_mH^m(\widetilde{\T}_0)$. Thus we have an exact sequence
\[
0\lra H^{m-2}(\Xi)\lra Gr_mH^m(\widetilde{\T}_0)\lra H^m(\Upsilon^\nu)\lra0
\]
and a noncanonical isomorphism
\[
Gr_mH^m(\widetilde{\T}_0)\cong H^{m-2}(\Xi)\oplus H^m(\Upsilon^\nu)\cong H^m(B)\oplus H^{m-2}(\Xi\cap\Xi_b)\oplus H^{m-2}(\Xi) 
\]
\end{proof}

\begin{corollary}\label{corGrthetat}The monodromy weight filtration on $H^m(\T_t)_{\lim}$ satisfies
\[
Gr_{m+1}H^m({\T}_t)_{\lim}\cong Gr_{m-1}H^m({\T}_t)_{\lim}\cong H^{m-1}(\Xi)
\]
and
\[
Gr_mH^m(\T_t)_{\lim}\cong H^m (B) \oplus H^{m-2}(\Xi\cap\Xi_b).
\]
\end{corollary}
\begin{proof} Analogous to the proof of Corollary \ref{corGrAt}.
\end{proof}

\subsection{The limit mixed Hodge structure on the primal cohomology}

The primal cohomology inherits a limit mixed Hodge structure from that of $H^4 (\T_t)$. We have

\begin{proposition}
Let $\bH\subset H^3(\Xi, \bZ)$ be the primal cohomology of $\Xi$ in $B$ and let $\bW \subset H^2 (\Xi\cap\Xi_b, \bZ)$ be the primal cohomology of $\Xi\cap\Xi_b$, i.e., $\bW := \Ker (j_*: H^2 (\Xi\cap\Xi_b, \bZ) \ra H^6 (B, \bZ))$. We have
\[
Gr_4 \bK_\bQ \cong \bW_\bQ, \quad \hbox{ and } \quad Gr_3 \bK_\bQ \cong \bH_\bQ.
\]
\end{proposition}
\begin{proof} Follows from the commutative diagram of Clemens-Schmid exact sequences
\begin{eqnarray}\label{eqnmixedK}
\xymatrix{\ar[r]&H^{4}(\widetilde{\T}_{V})\ar[r]^-{i_{t}^*}\ar[d]^-{j_*}&H^{4}(\T_{t})_{\lim}\ar[r]^-{N}\ar[d]^-{j_{t*}}&H^{4}(\T_{t})_{\lim}\ar[r]^-{\beta}\ar[d]^-{j_{t*}}&H_{6}(\widetilde{\T}_{V})\ar[d]^-{j_*}\ar[r]&\\
\ar[r]&H^{6}(\widetilde{A}_{V})\ar[r]&H^{6}(A_{t})_{\lim}\ar[r]&H^{6}(A_{t})_{\lim}\ar[r]&H_{6}(\widetilde{A}_{V})\ar[r]&}
\end{eqnarray}
and Corollaries \ref{corGrAt} and \ref{corGrthetat}.
\end{proof}

The action of the involution $y\mapsto b-y$ on $\Xi \cap \Xi_b$ splits $\bW_\bQ$ into the sum of its eigenspaces $\bW_\bQ^{+1}$ and $\bW_\bQ^{-1}$. By, e.g., \cite[pp. 3-4]{KramerWeissauer-theta} or \cite[pp. 7-8]{Kramer2013-fourfoldWeyl}, $\bW_\bQ^{+1}$ has dimension $52$ and Hodge numbers $w^{2,0} = w^{0,2} = 11, w^{1,1}= 30$, whereas $\bW_\bQ^{-1}$ consists of Hodge classes and has dimension $6$. Assuming $B$ sufficiently general, it follows from the main result of \cite{KramerWeissauer-theta} that the variations of Hodge structures formed by $\bW_\bQ^{+1}$ and $\bW_\bQ^{-1}$ as $b$ varies (see \cite{KramerWeissauer-theta}) are simple. We shall need

\begin{lemma}\label{lemW+simple}
For generic $B$ and $b\in B$, the Hodge structure $\bW_\bQ^{+1}$ is simple.
\end{lemma}
\begin{proof}
Let $B$ be generic and $U\subset B$ be the largest open subset on which the Hodge structures $\bW_\bQ^{+1}$ form a variation of Hodge structures. The simplicity of this variation of Hodge structures means that the corresponding representation $\widetilde{\pi}_1 (U) \ra GL (\bW_\bQ^{+1})$ of the algebraic monodromy group of $U$ is simple. Since the Mumford-Tate group $MT(\bW_\bQ^{+1})$ contains a finite index subgroup of the algebraic monodromy group of $U$ (see, e.g., \cite[Prop. 6]{Schnell-MumfordTate-2011}), it follows that, as representations of $MT(\bW_\bQ^{+1})$, we have $\bW_\bQ^{+1} \cong M \otimes \bW'$ where $\bW'$ is a simple representation of $MT(\bW_\bQ^{+1})$ and $M$ is its multiplicity space. This means that $\bW_\bQ^{+1} \cong M \otimes \bW'$ as Hodge structures, where $M$ is a trivial Hodge structure. Therefore the Hodge numbers of $\bW_\bQ^{+1}$ are multiples of those of $\bW'$ which is only possible if $\bW_\bQ^{+1} \cong \bW'$ is simple.
\end{proof}

\section{The action of $-Id$}\label{sec-Id}

\subsection{The dimensions of the eigenspaces} The action of $-Id$ leaves $\bK$ globally invariant and induces a splitting of $\bK_\bQ$
\[
\bK_\bQ = \bK_\bQ^{+1} \oplus \bK_\bQ^{-1}
\]
into the sum of its eigenspaces for $+1$ and $-1$ which are Hodge substructures of $\bK_\bQ$. We have

\begin{lemma}
The subspace $\bK^{g-2,1}$ is contained in $\bK_\bC^{(-1)^g}$. The eigenspace $(\bK^{g-3,2})^{(-1)^{g}} = \bK^{g-3,2} \cap \bK_\bC^{(-1)^{g}}$ has dimension $\frac{3^g+1}{2} - 2^g - {g \choose 2}^2 + g{g\choose 3} = \frac{3^g+1}{2} - 2^g - \frac{1}{6}{g^2 \choose 2}$.
\end{lemma}
\begin{proof}
By \cite[pp. 4-5]{IzadiWang2015}, we have the natural exact sequence
\[
0 \lra H^1 (\Omega^{g-1}_A)_{prim} \stackrel{m}{\lra} H^0 (\T, \omega_A |_\T (2\T)) \lra \bK^{g-2, 1} \lra 0
\]
where $H^1 (\Omega^{g-1}_A)_{prim} = \Ker (H^1 (\Omega^{g-1}_A) \ra H^2 (\omega_A))$ is the primitive cohomology. 
The automorphism $-Id$ of $A$ acts as the identity on $H^0 (\T, \cO_\T (2\T))$ and as $(-1)^g$ on $H^1 (\Omega^{g-1}_A)$ and on sections of $\omega_A$. Hence the action of $-Id$ on $\bK^{g-2,1}$ is via multiplication by $(-1)^g$.

By \cite[pp. 4-5]{IzadiWang2015}, we also have the exact sequences
\begin{eqnarray}\label{eqnHodge}
0\lra H^2(\Omega^{g-2}_A)_{prim}\lra \frac{F^{g-2}H^g(U)}{F^{g-1}H^g(U)}\lra \bK^{g-3,2}\lra0\\
\label{eqnHodge1} H^0(\T,\cO_\T(\T))\otimes H^0(\T,\omega_A |_\T(2\T)) \stackrel{m}{\lra} H^0(\T,\omega_A |_\T(3\T)) \lra \frac{F^{g-2}H^g(U)}{F^{g-1}H^g(U)}\lra0,
\end{eqnarray}
where $U = A\setminus \T$ and $m$ is again multiplication of sections. As before $-Id$ acts as multiplication by $(-1)^g$ on $H^2(\Omega^{g-2}_A)$ and $H^0(\T,\omega_A |_\T(2\T))$, and as the identity on $H^0(\T,\cO_\T(\T))$.

We claim that the eigenspaces of $-Id$ for $+1$ and $-1$ in $H^0(A,\cO_A(3\T))$ have dimensions $\frac{3^g +1}{2}$ and $\frac{3^g-1}{2}$ respectively. This can be checked, for instance, by induction on $g$ since these dimensions are constant on $\cA_g$. The claim is clearly true for $g=1$. Assuming it holds for $g-1$, one checks it for the product $E\times B$ of an elliptic curve and a ppav of dimension $g-1$:
\[
H^0 (\cO_{E\times B} (3\T)) = H^0 (\cO_E (3p))\otimes H^0 (\cO_B (3\Xi))
\]
where $p$ is a point of order 2 of $E$ and $\Xi$ is a symmetric theta divisor of $B$. The claim can now be checked with a simple calculation.

Since $H^0 (\cO_\T (3\T)) = H^0 (\cO_A (3\T)) / \theta \cdot H^0 (\cO_A (2\T))$ where $\theta$ is a nonzero section of $H^0 (\cO_A (\T))$, we see that the eigenspaces of $-Id$ for $+1$ and $-1$ in $H^0(\cO_\T(3\T))$ have dimensions $\frac{3^g +1}{2} - 2^g$ and $\frac{3^g-1}{2}$ respectively.

Now since the image of $m$ is contained in the eigenspace for $(-1)^{g-1}$, it follows from sequence \eqref{eqnHodge1} that the dimension of the eigenspace of $ \frac{F^{g-2}H^g(U)}{F^{g-1}H^g(U)}$ for the eigenvalue $(-1)^{g}$ is $\frac{3^g+1}{2} - 2^g$. Finally, since the action of $-Id$ on $H^2(\Omega^{g-2}_A)_{prim}$ is multiplication by $(-1)^g$, it follows from sequence \eqref{eqnHodge} that $(\bK^{g-3,2})^{(-1)^{g}}$ has dimension $\frac{3^g+1}{2} - 2^g - \dim H^2(\Omega^{g-2}_A)_{prim} = \frac{3^g+1}{2} - 2^g - {g \choose 2}^2 + g{g\choose 3}$.
\end{proof}

\begin{corollary}\label{corg-5}
The Hodge structure $\bK^{(-1)^{g-1}}$ has level $g-5$.
\end{corollary}

In particular, when $g=5$, we compute $\dim \bK_\bQ^{+1} =6$ and $\bK_\bQ^{+1}$ consists of Hodge classes, $\dim \bK_\bQ^{-1} = 72, \dim (\bK^{2,2})^{-1} = 40$.

\subsection{Action of the local monodromy} Since the limits of Hodge classes give Hodge classes in the limit mixed Hodge structures in sequence \ref{eqnmixedK}, we deduce that $N (\bK_\bQ^{+1}) = 0$ because $\bK_\bQ^{+1}$ consists of Hodge classes and $\bH$ is a simple Hodge structure by the main result of \cite{IzadiVanStraten}. In particular, $Gr_3 (\bK_\bQ^{+1}) =0$ and $Gr_3 (\bK_\bQ^{-1}) = Gr_3 (\bK_\bQ) = \bH_\bQ$.
Projecting further to $Gr_4 \bK_\bQ$ and then to $\bW_\bQ^{+1}$, it follows from Lemma \ref{lemW+simple} that $\bK_\bQ^{+1}$ maps isomorphically to $\bW_\bQ^{-1}$ and also that $Gr_4 \bK_\bQ^{-1}$ maps isomorphically to $\bW_\bQ^{+1}$.

\section{Global monodromy}\label{secglobalmono}

\subsection{The vanishing cocycles near the boundary}\label{subsecvan}
Let $Z\stackrel{\iota}\ra|H|\cong\bP^5$ be the $2$-to-$1$ cover ramified exactly along $\Gamma:=\cD+\overline{\cP_H^{-1}(N_0)}$ and set $X:=\iota^{-1}l$, $\cV:=Z\setminus\G$. Note that $Z$ exists since $\Gamma$ has even degree by Proposition \ref{propsingular} and Corollary \ref{corN0}, hence the line bundle $\cO_{\bP^5} (\Gamma)$ has a square root and the standard construction of ramified double covers applies. The curve $X$  is a 2-to-1 cover of $l$ ramified along $X\cap\G$. After making a base change to $X$ and blowing up the singular locus of each singular theta divisor, we obtain a family $(\widetilde{A},\widetilde{\T})$ with general fiber $(A_t,\T_t)$:
\[
\xymatrix{\T_t\ar[r]^-{i_t}\ar[d]_-{j_t}&\widetilde{\T}\ar[d]^-j\\
A_t\ar[d]\ar[r]^-{h_t}&\widetilde{A}\ar[d]^-{p}\\
\{t\}\ar[r]&X.}
\]
The total spaces of $\widetilde{A}$ and $\widetilde{\T}$ are smooth and the local pictures are described in Sections \ref{subsecnull}, \ref{subsectwonodes} and \ref{subsecbdy}.

For each $s_i$, $i=1,...,42$, corresponding to the degeneration in Section \ref{secpencil} (also see Section \ref{subsecbdy}), choose a small disk $V_i\ni s_i$ and pick a general point $t_i\in V_i$. Let $\gamma_i\subset X\cap\cV$ be a general path connecting $t$ with $t_i$. The family $\widetilde{\T}|_{\cup\gamma_i}$ deformation retracts to $\T_t$. Thus we have induced {\bf diffeomorphisms} 
\[
\psi_i: \T_{t}\lra \T_{t_i}.
\]
Over each $V_i$ we have the Clemens-Schmid exact sequences (\ref{eqCS}) for the degenerations of the abelian varieties and their theta divisors 
\begin{eqnarray}\label{eqnCS2}
\xymatrix{\ar[r]&H^{m}(\widetilde{\T}_{V_i})\ar[r]^-{i_{t_i}^*}\ar[d]^-{j_*}&H^{m}(\T_{t_i})_{\lim}\ar[r]^-{N_i}\ar[d]^-{j_{t_i*}}&H^{m}(\T_{t_i})_{\lim}\ar[r]^-{\beta_i}\ar[d]^-{j_{t_i*}}&H_{10-m}(\widetilde{\T}_{V_i})\ar[d]^-{j_*}\ar[r]&\\
\ar[r]&H^{m+2}(\widetilde{A}_{V_i})\ar[r]&H^{m+2}(A_{t_i})_{\lim}\ar[r]&H^{m+2}(A_{t_i})_{\lim}\ar[r]&H_{10-m}(\widetilde{A}_{V_i})\ar[r]&.}
\end{eqnarray}
Here $j_*:H^m(\widetilde{\T}_{V_i})\ra H^{m+2}(\widetilde{A}_{V_i})$ is defined to be the transpose of $j^*:H^{10-m}_c(\widetilde{A}_{V_i})\ra H^{10-m}_c(\widetilde{\T}_{V_i})$ under Poincar\'e duality and is a morphism of mixed Hodge structures \cite[Section 8]{IzadiTamasWang}.

Put $\mathbb{V}_i^m:=\psi^*_i\Ker\beta_i=\psi_i^*\Img(N_i)=\psi^*_iGr_{m-1}H^m(\T_{t_i})_{\lim}\subset H^m(\T_t)_{\lim}$. 

\begin{proposition}The space $\mathbb{V}_i$ is the space of `local vanishing $m$-cocycles', i.e., cohomology classes whose Poincar\'e dual vanishes in $\tT_{V_i}$. 
\end{proposition}
\begin{proof}This follows immediately from the definition of  $\beta_i$ in (\ref{eqnbeta}).
\end{proof}


By Corollary \ref{corGrthetat}, we have
\[
\Img(N_i)= Gr_{m-1}H^m(\T_{t_i})_{\lim}\stackrel{i_{t_i}^*}\cong Gr_{m-1}H^m(\widetilde{\T}_{V_i})\cong H^{m-1}(\Xi_i).
\]
When $m=4$, we can further rewrite the above isomorphisms as
\begin{eqnarray}\label{eqndecomp}
Gr_3H^4(\T_{t_i})_{\lim}\cong H^3(\Xi_i)\cong H^3(B_i)\oplus \bH'_i\cong j_{t_i}^*Gr_3H^4(A_{t_i})_{\lim}\oplus\bH'_i,
\end{eqnarray}
where $\bH'_i\subset H^3(\Xi_i)$ is the primal cohomology of $\Xi_i$ in $B_i$, which is $10$-dimensional. Let $\bH_i\subset\mathbb{V}_i^4\subset H^4(\T_t)$ be the image of $\bH_i'$ under the composition
\[
H^3(B_i)\oplus\bH'_i\cong Gr_3H^4(\T_{t_i})_{\lim}\subset H^4(\T_{t_i})_{\lim}\stackrel{\psi^*_{i}}\lra H^4(\T_t).
\]

Let $H^m(\T_t)_{var}:=\Ker(i_{t*}:H^m(\T_t)\ra H^{m+2}(\widetilde{\T}))$ and $H^m(A_t)_{var}:=\Ker(h_{t*}:H^m(A_t)\ra H^{m+2}(\widetilde{A}))$ be the variable cohomology of $\T_t$ in $\widetilde{\T}$ and $A_t$ in $\widetilde{A}$, respectively.

\begin{proposition}\label{propgenerate}The variable cohomology $H^m(\T_t)_{var}$ is equal to $\sum_{i=1}^{42}\mathbb{V}_i^m$.
\end{proposition}
\begin{proof} By  Equations \eqref{eqnLP} and \eqref{eqnmono}, when the theta divisor has one or two nodes, the local monodromy representation is trivial after we make a base change of order $2$. Thus from the Clemens-Schmid sequence, there are no `local vanishing cocycles' near these singular theta divisors. Therefore, by the global invariant cycles theorem, and since $\pi_1 (X \cap \cV)$ is generated by loops around the points of $X \setminus (X\cap \cV)$, the variable cohomology is generated by the `local vanishing cocyles' near $\T_{s_i}$, $i=1,...,42$. 
\end{proof}

\subsection{The proof of the main theorem} From now on we will abuse notation by considering $N_i$ in diagram \eqref{eqnCS2} as an endomorphism on $H^4(\T_t)$ via $\psi_i^*$ and then restricting it to $\bK_t^{-1}$. With the new notation, $N_i:\bK_t^{-1}\ra \bK_t^{-1}$ satisfies 
\begin{eqnarray}
&& N_i^2=0,\\
&& N_i(\bK_t^{-1})=\bH_i.
\end{eqnarray} 
Each $N_i$ induces a `limit mixed Hodge structure' $\bK^{i,-1}_{\lim}$ on $\bK_t^{-1}$ as in Section \ref{subseclim}.
We shall prove below that $\cap_i \Ker (N_i) = 0$ and deduce from it that $\sum \bH_i = \bK_t^{-1}$. We first need

\begin{lemma}\label{lemZari}Notation as in Section \ref{subsecvan}, the inclusion map induces a surjective map on fundamental groups
\[
\pi_1(X\cap\cV)\lra\pi_1(\cV).
\]
\end{lemma}
\begin{proof} By the construction of $Z$, the map $\iota:\cV\ra |H|\setminus\G$ is an \'etale  double cover. Thus we have the commutative diagram of exact sequences
\[
\xymatrix{1\ar[r]&\pi_1(X\cap\cV)\ar[r]^-{\iota_*}\ar[d]&\pi_1(l\setminus\G)\ar[d]\ar[r]&\bZ_2\ar[d]^-{\cong}\ar[r]&0\\
1\ar[r]&\pi_1(\cV)\ar[r]^-{\iota_*}&\pi_1(|H|\setminus\G)\ar[r]&\bZ_2\ar[r]&0.}
\]
By Zariski's theorem \cite[7.4.1]{Lamotke-Lefschetz-1981} the natural map
\[\pi_1(l\setminus\G)\lra \pi_1(|H|\setminus\G)
\]
is surjective if $l$ meets $\G$ transversely. The statement now follows from the diagram above.
\end{proof}
Denote $\bW_\bQ^{i,+1} := Gr_4 \bK^{i,-1}_{\lim}$ as in Lemma \ref{lemW+simple}. We can now prove
\begin{lemma}\label{lemker}We have $\cap_{i=1}^{42}\Ker (N_i)=0$.
\end{lemma}
\begin{proof}
Consider the family of theta divisors over $\cV$:
\[
\T_\cV \lra \cV.
\]
By the global invariant cycles theorem \cite[4.1]{deligne721}, for any smooth compactification $\overline{\T}$ of $\T_\cV$, the image of $H^4 (\oT)$ in $H^4 (\T_t)$ is the space of classes fixed by the monodromy action of $\pi_1 (\cV)$. Since, by Lemma \ref{lemZari}, the map $\pi_1 (X\cap\cV) \ra \pi_1 (\cV)$ is surjective, this is also equal to the space of classes fixed by the monodromy action of $\pi_1 (X\cap \cV)$ which is, again by the global invariant cycles theorem, also equal to the image of $H^4 (\tT)$ in $H^4 (\T_t)$. For all $i$, the map $H^4 (\tT) \ra H^4 (\T_t)$ factors through $H^4 (\tT_{V_i}) = H^4 (\T_{s_i})$. By the functoriality of the mixed Hodge structures on the cohomology of algebraic varieties, the image of the pull-back map $H^4 (\tT) \ra H^4 (\T_{s_i})$ maps injectively to $Gr_4 H^4 (\T_{s_i})$. Hence, by sequence \eqref{seqGR4}, the image of $H^4 (\tT) \ra H^4 (\T_t) = H^4 (\T_{t_i})$ maps injectively to $Gr_4 H_{lim}^4 (\T_{t_i})$.

The space $\cap_{i=1}^{42}\Ker (N_i)$ is the intersection of the image of $H^4 (\oT)$ or $H^4 (\tT)$ in $H^4 (\T_t)$ with $\bK_t^{-1}$. It therefore maps injectively to $Gr_4 \bK_t^{i,-1} = \bW^{i,+1}_\bQ$ for all $i$. Since $\bW^{i,+1}_\bQ$ is simple by Lemma \ref{lemW+simple}, it follows that either $\cap_{i=1}^{42}\Ker (N_i) =0$ or $\cap_{i=1}^{42}\Ker (N_i) = \bW^{i,+1}_\bQ$. To exclude the latter, move the pencil $l$ so that the isomorphism class of $\bW^{i,+1}_\bQ$ (as a Hodge structure) changes. This is possible because this Hodge structure does go to the boundary of its period domain when $\Xi \cap \Xi_b$ becomes singular.
\end{proof}

\begin{corollary}\label{corprim}
We have
\[
\sum \bH_i = \bK_t^{-1}.
\]
\end{corollary}
\begin{proof}
Since the monodromy operator preserves the intersection product $\langle,\rangle$ on $\bK_t^{-1}$, $N_i$ also satisfies the equality
\begin{eqnarray}\label{eqnN}
&&\langle N_i(x),y\rangle+\langle x,N_i(y)\rangle=0
\end{eqnarray} 
for any $x,y\in \bK_t^{-1}$.

Since the intersection pairing is nondegenerate on $\bK_t^{-1}$, equation (\ref{eqnN}) implies that $\Ker (N_i)^\perp = \bH_i$, $i=1,...,42$. Hence $\sum \bH_i = (\cap_i \Ker (N_i))^\perp = \bK_t^{-1}$.
\end{proof}

\begin{lemma}\label{lemconjugate}
With the notation of Section \ref{subsecvan}, all $\bH_i, i=1,...,42$ are conjugate under the monodromy representation
\[ \rho:\pi_1(X\cap\cV, t)\lra \mathrm{Aut}(\bK_t^{-1},\langle,\rangle).
\]
\end{lemma}
\begin{proof} For any $i\ne j$, choose a path $\delta'$ in $l$ connecting $t_i$ and $t_j$. By perturbing $\delta'$, we can assume $\delta'$ does not intersect the inverse image of $N_0$. We can lift $\delta'$ to a path $\delta\subset X\cap \cV$ as a smooth section over $\delta'$ in the tubular neighborhood of the smooth locus $\cD^0$ of $\cD$ in $\cV$.  A $\cC^\infty$-trivialization of the total space of the theta divisors over $\delta$ induces a map on cohomology, which sends $\bH'_i\subset H^4(\T_{t_i})$ to $\bH_j'\subset H^4(\T_{t_j})$. This precisely means that under the monodromy action, $\rho(\gamma_i\cdot\delta\cdot\gamma_j^{-1})$ sends $\bH_i$ to $\bH_j$.  
\end{proof}

\subsection{\em Proof of {\bf Theorem} \ref{thmmain}} It suffices to show that for very general $t\in \cV$, $\bK_t^{-1}$ is a simple Hodge structure. Suppose $0\subsetneqq\mathbb{F}_t\subset \bK_t^{-1}$ is a rational Hodge substructure, then $\mathbb{F}_t$ is an invariant subspace under the action of the Mumford-Tate group $MT(\bK_t^{-1})$. For very general $t$, $MT(\bK_t^{-1})$ contains the identity component $I_{\cV}$ of the {\bf algebraic monodromy group} $G_{\cV}$, i.e., the Zariski closure in $GL(\bK_t^{-1})$ of the monodromy group $\rho(\pi_1(\cV))$, (c.f. \cite[Prop. 6]{Schnell-MumfordTate-2011}).

Now note that $T_i\in G_{\cV}$ is unipotent. In characteristic zero, any nontrivial unipotent element of an algebraic group generates a subgroup with connected Zariski closure $\bG_a$. This implies that $T_i \in I_{\cV}$. Therefore $\bF_t$ is invariant under $T_i$, hence also $N_i$. Therefore $\bF_t$ is the fiber of a sub-variation of Hodge structures of the variation given by $\bK_t^{-1}$ on $X\cap \cV$ and each $N_i$ then induces a `limit mixed Hodge structure' $\bF_{\lim}^i$ on $\bF_t$.

By Lemma \ref{lemker}, for any $0\ne x\in\mathbb{F}_t$, $x\notin \Ker (N_i)$ for some $i$, thus $0\ne N_i(x)\in \mathbb{F}_t\cap \mathbb{H}_i=\mathbb{F}_t\cap W_3\bK^{i,-1}_{\lim}=W_3\bF_{\lim}^i$. Since $\bH_i=W_3\bK^{i,-1}_{\lim}$ is a simple pure Hodge structure (follows from the main result of \cite{IzadiVanStraten}), we conclude $\bH_i\subset\mathbb{F}_t$. By Lemma \ref{lemconjugate}, the $\bH_i$ are conjugate under the monodromy group $\pi_1(\cV)$, thus $\bH_i \subset \bF_t$ for all $i$ and, by Corollary \ref{corprim}, $\mathbb{F}_t=\bK_t^{-1}$.
\qed


\providecommand{\bysame}{\leavevmode\hbox to3em{\hrulefill}\thinspace}
\providecommand{\MR}{\relax\ifhmode\unskip\space\fi MR }
\providecommand{\MRhref}[2]{%
  \href{http://www.ams.org/mathscinet-getitem?mr=#1}{#2}
}
\providecommand{\href}[2]{#2}


\begin{thebibliography}{FGSMV14}

\bibitem[ACG11]{ACG}
E.~Arbarello, M.~Cornalba, and P.A. Griffiths, \emph{Geometry of algebraic
  curves, volume {II} with a contribution by {J}.{D}. {H}arris}, Springer,
  Berlin, 2011.

\bibitem[AFS15]{ArbarelloFerrettiSacca}
E.~Arbarello, A.~Ferretti, and G.~Sacca, \emph{Relative {P}rym varieties
  associated to the double cover of an {E}nriques surface}, J. Differential
  Geom. \textbf{100} (2015), no.~2, 191--250.

\bibitem[Bea77]{Beauville1977-1}
A.~Beauville, \emph{Prym varieties and the {S}chottky problem}, Inventiones
  Math. \textbf{41} (1977), 149--196.

\bibitem[BHPdV04]{BHPV}
W.~P. Barth, K.~Hulek, C.~A.~M. Peters, and A.~Van de~Ven, \emph{Compact
  complex surfaces}, Ergebnisse der Mathematik und ihrer Grenzgebiete. 3.
  Folge. A Series of Modern Surveys in Mathematics, vol.~4, Springer, Berlin,
  2004.

\bibitem[CDI]{ConderDeweyIzadi}
J.~Conder, E.~Dewey, and E.~Izadi, \emph{Surfaces generating the invariant
  primal cohomology of a four dimensional theta divisor}, In preparation.

\bibitem[Cle77]{clemens77}
H.~Clemens, \emph{Degeneration of {K}\"ahler manifolds}, Duke Math. J.
  \textbf{44} (1977), no.~2, 215--290.

\bibitem[Cor89]{Cornalba89}
M.~Cornalba, \emph{Moduli of curves and theta-characteristics}, Lectures on
  {R}iemann surfaces ({T}rieste, 1987), World Sci. Publ., Teaneck, NJ, 1989,
  pp.~560--589.

\bibitem[Deb92]{Debarre921}
O.~Debarre, \emph{Le lieu des vari\'et\'es ab\'eliennes dont le diviseur
  th\^eta est singulier a deux composantes}, Ann. Sci. \'{E}cole Norm. Sup. (4)
  \textbf{25} (1992), no.~6, 687--707.

\bibitem[Del72]{deligne721}
P.~Deligne, \emph{Th\'eorie de {H}odge {II}}, Publ. Math. I.H.E.S. \textbf{40}
  (1972), 5--58.

\bibitem[DS81]{DonagiSmith81}
R.~Donagi and R.~Smith, \emph{The structure of the {P}rym map}, Acta Math.
  \textbf{146} (1981), 25--102.

\bibitem[FGSMV14]{FarkasGrushevskySalvatiManniVerra}
G.~Farkas, S.~Grushevsky, R.~Salvati~Manni, and A.~Verra, \emph{Singularities
  of theta divisors and the geometry of ${{\mathcal A}_5}$}, J. Eur. Math. Soc.
  (JEMS) \textbf{16} (2014), no.~9, 1817--1848.

\bibitem[Igu67]{Igusa-DesingSieg-1967}
J.~Igusa, \emph{A desingularization problem in the theory of {S}iegel modular
  functions}, Math. Ann. \textbf{168} (1967), 228--260.

\bibitem[IS95]{IzadiVanStraten}
E.~Izadi and D.~van Straten, \emph{The intermediate jacobians of the
  theta-divisors of four-dimensional principally polarized abelian varieties},
  Journal of Algebraic Geometry \textbf{4} (1995), no.~3, 557--590.

\bibitem[ITW17]{IzadiTamasWang}
E.~Izadi, Cs. Tamas, and J.~Wang, \emph{The primitive cohomology of the theta
  divisor of an abelian fivefold}, J. Algebraic Geom. \textbf{26} (2017),
  107--175.

\bibitem[IW15]{IzadiWang2015}
E.~Izadi and J.~Wang, \emph{The primitive cohomology of theta divisors},
  Proceedings of the conference on {H}odge theory and classical algebraic
  geometry, Contemp. Math., vol. 647, AMS, Providence, RI, 2015, pp.~79--90.

\bibitem[Kr]{Kramer2013-fourfoldWeyl}
T.~Kr\"amer, \emph{On a family of surfaces of general type attached to abelian
  fourfolds and the {W}eyl group ${W}({E}_6)$}, arXiv:1309.4739v1.

\bibitem[KS88]{KempfSchreyer1988}
G.~R. Kempf and F.~O. Schreyer, \emph{A {T}orelli theorem for osculating cones
  to the theta divisor}, Compositio Mathematica \textbf{67} (1988), 343--353.

\bibitem[KW]{KramerWeissauer-theta}
T.~Kr\"amer and R.~Weissauer, \emph{The symmetric square of the theta divisor
  in genus 4}, arXiv:1109.2249v2.

\bibitem[Lam81]{Lamotke-Lefschetz-1981}
K.~Lamotke, \emph{The topology of complex projective varieties after {S}.
  {L}efschetz}, Topology \textbf{20} (1981), 15--51.

\bibitem[Lan73]{Landman1973}
A.~Landman, \emph{On the {P}icard-{L}efschetz transformation for algebraic
  manifolds acquiring general singularities}, Trans. Amer. Math. Soc.
  \textbf{181} (1973), 89--126.

\bibitem[MM83]{MoriMukai83}
S.~Mori and S.~Mukai, \emph{The uniruledness of the moduli space of curves of
  genus 11}, Algebraic geometry (Tokyo/Kyoto, 1982), Lecture Notes in Math.,
  vol. 1016, Springer, Berlin, 1983, pp.~334--353.

\bibitem[Mor84]{morrison84}
D.~R. Morrison, \emph{The {C}lemens-{S}chmid exact sequence and applications},
  Topics in Transcendental Algebraic Geometry (P.~A. Griffiths, ed.), Annals of
  Math. Stud., vol. 106, Princeton Univ. Press, Princeton, NJ, 1984,
  pp.~101--119.

\bibitem[Mum74]{Mumford-PrymI-1974}
D.~Mumford, \emph{Prym varieties {I}}, Contributions to Analysis (L.V. Ahlfors,
  I.~Kra, B.~Maskit, and L.~Niremberg, eds.), Academic Press, 1974,
  pp.~325--350.

\bibitem[Mum75]{Mumford1975}
\bysame, \emph{Curves and their jacobians}, The University of Michigan Press,
  Ann Arbor, Mich., 1975.

\bibitem[Mum83]{MumfordKodAg-1983}
\bysame, \emph{On the {K}odaira dimension of the {S}iegel modular variety},
  Algebraic Geometry- Open problems (Ravello 1982), Lecture Notes in
  {M}athematics, vol. 997, Springer, 1983, pp.~348--375.

\bibitem[Nam85]{Namikawa-Enriques-1985}
Y.~Namikawa, \emph{Periods of {E}nriques surfaces}, Math. Ann. \textbf{270}
  (1985), no.~2, 201--222.

\bibitem[Sch73]{schmid73}
W.~Schmid, \emph{Variation of {H}odge structure: the singularities of the
  period mapping}, Inventiones Math. \textbf{22} (1973), 211--319.

\bibitem[Sch11]{Schnell-MumfordTate-2011}
C.~Schnell, \emph{Two lectures about {M}umford-{T}ate groups}, Rend. Semin.
  Mat. Univ. Politec. Torino \textbf{69} (2011), no.~2, 199--216.

\bibitem[Ste76]{Steenbrink1975}
J.~H.~M. Steenbrink, \emph{Limits of {H}odge structures}, Inventiones Math.
  \textbf{31} (1975/76), no.~3, 229--257.

\bibitem[SV85]{SmithVarley1985-A5}
R.~Smith and R.~Varley, \emph{Components of the locus of singular theta
  divisors of genus 5}, Algebraic geometry, Sitges (Barcelona), 1983
  (G.~E.~Welters E.~Casas-Alvero and S.~Xambó-Descamps, eds.), Lecture Notes in
  Math., vol. 1124, Springer, Berlin, 1985, pp.~338--416.

\bibitem[SV06]{SmithVarley2006-Pfaffian}
\bysame, \emph{The pfaffian structure defining a prym theta divisor}, The
  geometry of Riemann surfaces and abelian varieties (G.~E.~Welters
  E.~Casas-Alvero and S.~Xambó-Descamps, eds.), Contemp. Math., vol. 397, Amer.
  Math. Soc., Providence, RI, 2006, pp.~215--236.

\bibitem[Voi92]{Voisin1992-Wahl}
C.~Voisin, \emph{Sur l'application de {W}ahl des courbes satisfaisant la
  condition de {B}rill-{N}oether-{P}etri}, Acta Math. \textbf{168} (1992),
  no.~3-4, 249--272.

\end{thebibliography}
\end{document}